\documentclass[11pt,reqno]{amsart}
\usepackage{latexsym}
\usepackage{amsmath}
\usepackage{amssymb, amscd, amsthm, mathtools}
\usepackage[all]{xy}
\usepackage{multirow}
\usepackage{xcolor} 
\usepackage{tikz}
\usepackage{tikz-cd}
\usepackage[shortlabels, inline]{enumitem}
\SetEnumerateShortLabel{a}{\textup{(\alph*)}}
\SetEnumerateShortLabel{A}{\textup{(\Alph*)}}
\SetEnumerateShortLabel{1}{\textup{(\arabic*)}}
\SetEnumerateShortLabel{i}{\textup{(\roman*)}}
\SetEnumerateShortLabel{I}{\textup{(\Roman*)}}

\usepackage{mathrsfs}

\newtheorem{thm}{Theorem}[section]
\newtheorem{cor}[thm]{Corollary}
\newtheorem{lem}[thm]{Lemma}
\newtheorem{lemma}[thm]{Lemma}
\newtheorem{prop}[thm]{Proposition}

\theoremstyle{definition}
\newtheorem{defn}[thm]{Definition}

\theoremstyle{remark}

\newcommand{\F}{\ensuremath{\mathbb F}}
\newcommand{\gH}{\mathcal{H}}
\newcommand{\Z}{\ensuremath{\mathbb Z}}

\newcommand{\gO}{\mathcal{O}}
\newcommand{\gA}{\mathcal{A}}
\newcommand{\Q}{\ensuremath{\mathbb Q}}
\newcommand{\gG}{\ensuremath{\mathcal G}}
\newcommand{\gE}{\ensuremath{\mathcal E}}
\newcommand{\gJ}{\ensuremath{\mathcal J}}

\DeclareMathOperator{\End}{\mathrm{End}}
 
\DeclareMathOperator{\Aut}{\mathrm{Aut}}

\DeclareMathOperator{\Nrd}{\mathrm{Nrd}}

\DeclareMathOperator{\NS}{\mathrm{NS}}
\DeclareMathOperator{\Pic}{\mathrm{Pic}}
\DeclareMathOperator{\im}{\mathrm{im}}

\DeclareMathOperator{\Iso}{\mathrm{Iso}}

\DeclareMathOperator{\GL}{\mathrm{GL}}

\newcommand{\va}{\varphi}
\newcommand{\rw}{\rightarrow}

\newcommand{\su}{\subseteq}
\newcommand{\la}{\langle}
\newcommand{\ra}{\rangle}
\newcommand{\be}{\beta}
\newcommand{\al}{\alpha}
\newcommand{\overbar}[1]{\mkern 1.5mu\overline{\mkern-1.5mu#1\mkern-1.5mu}\mkern 1.5mu}

\newcommand{\ppav}{\mathrm{PPAV}}

\newcommand{\psas}{\mathrm{PPSSAS}}
\newcommand{\psav}{\mathrm{PPSSAV}}

\newcommand{\Na}{\mathbf{N1}}
\newcommand{\Nb}{\mathbf{N2}}
\newcommand{\Nc}{\mathbf{N3}}
\newcommand{\Nd}{\mathbf{N4}}
\newcommand{\Da}{\mathbf{D1}}
\newcommand{\Db}{\mathbf{D2}}
\newcommand{\Dc}{\mathbf{D3}}
\newcommand{\Dd}{\mathbf{D4}}
\newcommand{\De}{\mathbf{D5}}
\newcommand{\DD}{\mathbf{D}}
\newcommand{\NNN}{\mathbf{N}} 
\begin{document}

	\title[Neighborhood of vertices in the isogeny graph of $\psas$]{Neighborhood of vertices in the isogeny graph of principally polarized  superspecial abelian surfaces}
	\author{Zheng Xu$^{1}$} \author{Yi Ouuyang$^{1,2}$$^*$} \author{Zijian Zhou$^{3}$}
	
	\address{$^1$Hefei National Laboratory, University of Science and Technology of China, Hefei 230088, China}

	\address{$^2$School of Mathematical Sciences, Wu Wen-Tsun Key Laboratory of Mathematics,   University of Science and Technology of China, Hefei 230026, Anhui, China}

	\address{$^{3}$College of Sciences, National University of Defense Technology,	Changsha 410073, Hunan, China}
	\email{xuzheng1@mail.ustc.edu.cn}
	\email{yiouyang@ustc.edu.cn}
	\email{zhouzijian122006@163.com}

	
	\date{}
	\maketitle

	\begin{abstract}
		For two supersingular elliptic curves $E$ and $E'$ defined over $\F_{p^2}$, let $[E\times E']$ be the superspecial abelian surface with the principal polarization $\{0\}\times E'+E\times \{0\}$.  We determine local structure of the  vertices $[E\times E']$ in the   $(\ell,\ell)$-isogeny graph of principally polarized superspecial abelian surfaces where either $E$ or $E'$ is defined over $\F_p$.  We also present a simple new proof of the main theorem in \cite{LOX20}.
	\end{abstract}

	\section{Introduction}
	The supersingular isogeny-based cryptography is a relatively new suggestion for post quantum cryptosystems and is based on the assumption that it  is computationally hard even for a quantum computer to find a path in the $\ell$-isogeny  graph between two given vertices or equivalently to compute the endomorphism rings of supersingular elliptic  curves. 
	Known instantiations  for the isogeny-based cryptography include the key exchange protocols SIDH \cite{jf}, CSIDH \cite{clm} and OSIDH \cite{ck}, and the signature schemes  GPS-sign \cite{gps} and SQI-sign \cite{fkl}.
	
	As isogenies between supersingular curves found success in post quantum cryptography, it is natural to generalize the isogeny-based cryptosystem to abelian varieties of higher dimensions.  Flynn-Ti \cite{ft}  constructed a SIDH-like key exchange cryptosystem based on supersingular hyperelliptic curves over finite fields. In \cite{cds}, Castryck-Decru-Smith  constructed hash functions by Richelot isogenies between superspecial abelian varieties of dimension two as a generalization of hash functions in \cite{cgl}.
	
	The  security of isogeny-based cryptosystems depends on finding paths between vertices in the supersingular isogeny graphs, thus it is important to study  the structure of these  graphs, especially  the neighbors of a fixed vertex.  
	
	In the supersingular elliptic curve case, the isogeny graph
	over $\F_p$ was constructed in \cite{dg}. 
	Adj-Ahmadi-Menezes \cite{aam} described the subgraphs with trace $0$ or $\pm p$.  The authors  (with S. Li) clarified the local structure of the isogeny graph at $\F_p$-vertices (i.e. whose $j$-invariants are inside $\F_p$) in \cite{OX,LOX201,LOX20}. As there is no effective way to describe explicitly  the endomorphism rings of supersingular elliptic curves defined over  $\F_{p^2}\backslash \F_p$, it is still an open problem to determine the local structure at these vertices.

	Now let us move on to the higher dimension case, in particular the abelian surfaces case (the dimension $2$ case). Ionica and Thom\'e \cite{it}  described the structure of the isogeny graph of ordinary abelian surfaces. Katsura and Takashima  \cite{kt1, kt2, t} obtained the number of Richelot isogenies from superspecial  abelian surfaces to products of supersingular curves. Florit and  Smith \cite{fs, fs1}  classified the Richelot isogenies  between  superspecial abelian surfaces. In \cite{jz1, jz2}, Jordan and Zaytman summarized the relationship between matrices of maximal orders of quaternion algebras and isogenies of principally polarized superspecial abelian varieties ($\psav$ in short), defined the big and little isogeny graphs of $\psav$ and studied their structures.    
	Moreover, the method to find a path and the concept of multiradical isogenies have been extended to the case of abelian surfaces in \cite{cs, cd2}.
	
	Note that the recent broken of SIDH \cite{cd1, MM, R} is based on the fact that the secret key can be recovered by  constructing an isogeny from  an abelian surface with information extracted from the given torsion points. This indicates that understanding the isogeny (graph) of  abelian surfaces  could also be very useful in studying the supersingular isogeny cryptosystems.
	
	For two supersingular elliptic curves $E$ and $E'$ defined over $\F_{p^2}$, let $[E\times E']$ be the superspecial abelian surface with the principal polarization $\{0\}\times E'+E\times \{0\}$. In this article, we determine local structure of the  vertices $[E\times E']$ in the  $(\ell,\ell)$-isogeny graph of principally polarized superspecial abelian surfaces where either $E$ or $E'$ is defined over $\F_{p}$. Previously, Florit and Smith\cite{fs, fs1} partitioned vertices in the graph into several types and  computed the number of edges connecting to each type in the $(2,2)$-isogeny graph. Here we treat the general case $\ell\geq 2$  and obtain the  numbers of edges connecting to any adjacent vertex rather than its type and also determine the structure of the associated isogeny kernels. We obtain our results by classifying the action of automorphism groups in the kernels of isogenies. 
	
	The paper is organized as follows. In \S 2 we recall basic results about principally polarized  superspecial  abelian varieties/surfaces and the associated $(\ell,\ell)$-isogeny graphs. We study the loops and neighbors of $[E\times E']$ in the cases (i) $E=E'=E_{1728}$, (ii) $E=E'=E_0$, (iii) $j(E)\in \F_p\backslash\{0,1728\}$ and $j(E')\neq j(E)$, (iv) $j(E)\in \F_p\backslash\{0,1728\}$ and $E'= E$ in \S 3-\S 6 respectively. We present   a simple new proof of the main theorem in \cite{LOX20} in \S 7.

	\section{Preliminaries}
	
	Throughout this paper we assume that $p$ is a prime number and 
	$k$ is a finite field of characteristic $p$. Let $\Bar{\mathbb{F}}_p$ be an algebraic closure of $k$. We assume the abelian varieties are defined over $\Bar{\mathbb{F}}_p$. 
	
	\subsection{Principally polarized abelian varieties}
	We recall a few facts about principally polarized abelian varieties. 
	
	Let $A$ be an abelian variety defined over $k$. Then a divisor $D$ determines an isogeny $\lambda_D: A \to \hat{A}$, the dual abelian variety of $A$. If $D$ is an ample divisor, then $\lambda_D$ is a polarization on $A$ .  If moreover $\deg(\lambda_D)=1$, then  $\lambda_D$  is   a principally polarization of $A$ and  $\gA= (A,D)$ is called a principally polarized abelian variety ($\ppav$ in short).

	For $m\in \Z_+$, let $A[m]$ be the $m$-torsion subgroup of $A$. 
	If $m$ is  prime to $p$, a subgroup $S$ of $A[m]$ is called maximal $m$-isotropic if it is maximal among subgroups $T$ of $A[m]$ such that the restriction of the Weil pairing $e_m: A[m]\times A[m]\rightarrow \mu_m$ on $T\times T$ is trivial. 
	
	The following result in \cite{mu} implies that the kernel of an isogeny of  principally polarized abelian varieties is a   maximal isotropic subgroup:
	\begin{thm} \label{theorem:3. 13}
		Let $(A, D)$ be  a principally polarized abelian variety over $\overbar{\F}_p$ and $S$ be a subgroup of $A[m]$. Denote by $\phi:A \to A'=A/S$ the isogeny with kernel $S$. Then there exists a principally polarized divisor $D'$ of $A'$ such that $\phi^{*}D'\sim mD$ if and only if $S$ is a maximal  $m$-isotropic subgroup. Particularly, if $S$ is a maximal  $m$-isotropic subgroup of $A[m]$, then        $(A', D')$ is also a principally polarized abelian variety.
	\end{thm}
	
	By the above result, if  $\va$ is an isogeny between a principally polarized abelian variety $(A, D)$ and an abelian variety $A'$, and $\ker(\va)$ is a maximal $m$-isotropic subgroup, then $A'$  also has a structure of principal polarization. 
	
	Suppose $p>3$ and  $\ell$ is a prime different from $p$. 
	The following result in \cite{ft} presents the types of maximal $\ell^n$-isotropic subgroups.
	\begin{prop} \label{prop:lniso}
		Let $\gA= (A, D)$ be a principally polarized abelian surface. Then there are two types of  maximal $\ell^n$-isotropic subgroups in $A[\ell^n]$:
		\begin{enumerate}[1]
			\item $\mathbb{Z}/\ell^n\mathbb{Z}\times \mathbb{Z}/\ell^n\mathbb{Z}$,
			\item $\mathbb{Z}/\ell^n\mathbb{Z}\times \mathbb{Z}/\ell^{n-k}\mathbb{Z}\times \mathbb{Z}/\ell^{k}\mathbb{Z}$ with $0\leq k\leq \lfloor\frac{n}{2}\rfloor$.
		\end{enumerate}
		The number of maximal $\ell^n$-isotropic subgroup  is equal to $\ell^{2n-3}(\ell^2+1)(\ell+1)(\ell^n+ \frac{\ell^{n-1}-1}{\ell-1})$.
		
		In particular, for $n=1$, there are $(\ell+1)(\ell^2+1)$ maximal $\ell$-isotropic subgroups,  all  of the form $\Z/\ell\Z\times \Z/\ell\Z$.
	\end{prop}
	
	\subsection{Superspecial abelian varieties}
	A supersingular abelian variety $A$ over  $\Bar{\F}_p$ is an abelian variety isogenous to a product of supersingular elliptic curves.  A superspecial abelian variety $A$ is  a supersingular abelian variety that is isomorphic to a product of supersingular elliptic curves.  
	
	By definition, superspecial = supersingular in the elliptic curve case. There are many (if $p$ is large enough) non-isomorphic supersingular elliptic curves over  $\Bar{\F}_p$.  However,  there is only one (up to isomorphism) superspecial abelian variety for each dimension $g>1$ by the following famous result:
	\begin{thm}[Deligne, Ogus\cite{Ogus}, Oort, Shioda\cite{Shioda}] \label{thm:doos} Any 
		superspecial abelian variety $A/\Bar{\F}_p$ of dimension $g>1$ is isomorphic to $E^g$, where $E$ is an arbitrary supersingular elliptic curve over $\Bar{\F}_p$. 
	\end{thm}
	
	\subsection{Principally polarized superspecial abelian varieties} Assigning a principal polarizations to the superspecial abelian variety (when the dimension $g>1$ is fixed), we obtain a principally polarized superspecial abelian variety ($\psav$ in short) and in particular a principally polarized superspecial abelian surface ($\psas$ in short) when $g=2$. Different principal polarization gives different $\psav$.	
	For example, the product $E_1 \times E_2 \times \cdots \times E_g$ of supersingular elliptic curves,   with the principal polarization $\{0\}\times E_2\times \cdots \times E_g+ \cdots + E_1\times E_2\times \cdots \times \{0\}$, is a $\psav$. 
	
	For abelian surfaces, the following is well-known:
	\begin{thm}
		There are two types of $\psas$ over  $\Bar{\F}_p$:
		\begin{enumerate}[1]
			\item Jacobian type $\gJ_p$, consisting of  Jacobians of  superspecial hyperelliptic curve of genus $2$ with  the canonical principal polarization, whose number    is
			\[ 	\#\gJ_p= \begin{cases}			
				0, & \text{ if }p=2,3,\\
				1, &  \text{ if }p=5,\\
				\dfrac{p^3+24p^2+141p-346}{2880}, & \text{ if } p>5.
			\end{cases}\]
			
			\item Product type $\gE_p$: consisting of products of two supersingular elliptic curves with the above principal polarization, whose number is
			\[ \# \gE_p=\begin{cases}			
				1,  &  \text{ if }p=2,3,5,\\ \frac{1}{2}{S_{p^2}}({S_{p^2}}+1), & \text{ if } p>5,			
			\end{cases}\]
			where $S_{p^2}$ is the number of isomorphism classes of supersingular elliptic curves over $\Bar{\F}_p$.
		\end{enumerate}
	\end{thm}
	
	\subsection{Relationship between isogenies and matrices} 
	Let $E$ be a fixed supersingular elliptic curve. It is well-known that  $\gO:=\End(E)$ is a maximal order in the quaternion algebra $B_{p,\infty}$ over $\Q$ ramified only at $p$ and $\infty$. 
	
	Suppose $g>1$. Then $E^g$ is a superspecial abelian variety of dimension $g$, equipped with the principal polarization $\{0\}\times E^{g-1}+\cdots + E^{g-1}\times \{0\}$. We have 
	$\End(E^g)=M_g(\gO)$ and
	\[ \Aut(E^g)=\GL_g(\gO)=\{M\in M_g(\gO)\mid M \ \text{is invertible}\}. \] 
	The reduced norm $\Nrd: \gO \rw \Z$ induces the reduced norm $\Nrd: M_g(\gO) \rw \Z$. Then we also have
	\[\GL_g(\gO)= \{M\in M_g(\gO)\mid \Nrd(M)= 1\}.\]

	Let $A$ be a superspecial abelian variety of dimension $g$.
	By Theorem~\ref{thm:doos}, $E^g$ and $A$ are isomorphic. Let $\iota_A: A\rw E^g$ be a fixed isomorphism which induces $\iota_A: \End(A)\cong M_g(\gO)$.  Note that another isomorphism $\iota_A'$ is uniquely determined by $\iota_A'\iota_A^{-1}\in \GL_g(\gO)$.  
	
	For   $M\in M_g(\gO)$, let $M^+$ denote the conjugate transpose of $M$. If $M$ is
	associated to the endomorphism $\alpha\in \End(A)$, then $M^{+}$ is the matrix associated to the Rosati involution  $\alpha^{\dagger}$ of $\alpha$. 
	
	Suppose $X$ is a principal polarized divisor of $A$. The map
	\[ \Pic(A)\rightarrow \End(A),\quad L\mapsto \lambda_{X}^{-1} \circ\lambda_L \]
	factors through the N\'eron-Severi group $\NS(A)= \Pic(A)/{\Pic^{0}(A)}$. Let
	\begin{equation}	j: \NS(A)  \rightarrow {\End} (A) \cong M_g(\gO); \qquad \overline{L} \rightarrow  \iota_A(\lambda_{X}^{-1} \circ \lambda_L).		\end{equation}  
	This map extends to  $j: \NS(A)\otimes \Q	\rw \End(A) \otimes \Q \cong M_g(\gO)\otimes \Q$. 
	
	\begin{prop}\cite[Proposition 14.2]{m}  \label{rosati1} The map $j$ is invariant under the Rosati involution, which implies that
		\[ j(\bar{L})=j(\bar{L})^{+}. \]
	\end{prop}

	The following result allows us to determine whether a divisor of an abelian variety corresponds to a (principal) polarization.
	\begin{prop}\cite[Proposition 2.8]{iko}
		Let $L$ be a divisor of an abelian variety $A$ of dimension $g$. Then
		\[\frac{L^g}{g!}=\chi(L)=\Nrd(j(\bar{L})), \qquad \chi(L)^2=\deg(\lambda_L),\] 
		and 
		\begin{enumerate}[1]
			\item $L$ is  associated to  a polarization (i.e. $L$ is an ample divisor) if and only if   $j(\bar{L})$ is positive definite;
			
			\item  $L$ is associated to a principal polarization if and only if   $j(\bar{L})$ is positive definite with reduced norm $1$. 
		\end{enumerate}
	\end{prop}
	
	Different choices of the isomorphism $\iota_A$ give different $j(\Bar{L})\in M_g(\gO)$, but they are related by the following result: 
	\begin{prop}\cite[Proposition 31]{jz1} Let \[\gH= \{H \in M_n(\gO) \mid H \ \text{is positive-definite Hermitian of reduced norm} \ 1\}.\] 
		Two matrices $H$ and $H'$ in $\gH$ correspond to the same polarized divisor if and only if they are in the same orbit 
		under the action of $\GL_g(\gO)$ on the set $\gH$:
		\[ \GL_g(\gO) \times \gH \rightarrow \gH;\ \ (M, H) \mapsto M^{+}HM. \]
		Moreover, there is a one-to-one correspondence between $\gH\big/\GL_g(\gO)$ and  the set of isomorphism classes of  $\psas$ of dimension $g$.
	\end{prop}	
	
	Applying the above results to the case $g=2$, we have
	\begin{prop}\cite[Corollary 2.9]{iko}\label{corre_L_M}
		For $g=2$ and  $d\in \Z_+$, there is a one-to-one correspondence 
		\begin{align*}
			\{ \bar{L}\in \NS(A) \mid L>0, L^2=2d\} &\rightarrow \left\{ \begin{pmatrix} a & b \\ \bar{b} & c \end{pmatrix} \in M_2({\gO}) \mid a,c \in \Z_+, ac-b\bar{b}=d \right \}\\
			\bar{L} \qquad &\mapsto \qquad j(\bar{L}).
		\end{align*}
	\end{prop}

	\subsection{The $(\ell, \ell)$-isogeny graph of $\psas$}
	From now on, suppose $p>3$ and  $\ell$ is a prime different from $p$. Suppose $g=2$, i.e. we are dealing with abelian surfaces.
	
	Let $\gA_1= (A, D_1)$ and $\gA_2= (A, D_2)$ be two principally polarized abelian surfaces over $\overline{\F}_p$. An $(\ell, \ell)$-isogeny is an isogeny $\phi:\gA_1\to \gA_2$  such that $\ker(\phi)=\Z/{\ell \Z} \times \Z/{\ell \Z}$. For an $(\ell, \ell)$ isogeny $\phi$, there exists a dual isogeny $\hat{\phi}:\gA_2\rightarrow \gA_1$ such that $\hat{\phi} \circ \phi=[\ell]$.
	
	We can describe $(\ell,\ell)$ isogenies using  matrices in $M_2(\gO)$ in the following proposition, whose proof  was given in \cite[Proposition 31]{jz1}. 
	\begin{prop} Let $A$ be a superspecial abelian surface, $P_1$ and $P_2$ be two principal polarizations of $A$. Let $H_1=j(\Bar{P_1})$ and $H_2=j(\Bar{P_2})$. If $\alpha : A \rightarrow A$ is an isogeny of degree $\ell^{2m}$ associated to  $M\in M_2(\gO)$, then  
		$\al^*(P_2)= \ell^m P_1$  if and only if $M^{+}H_2M=\ell^m H_1$, and in this case, $\al$ is an isogeny from $(A, P_1)$ to $(A, P_2)$.  
	\end{prop}
	
	\begin{defn}
		The $(\ell, \ell)$-isogeny graph of principally polarized superspecial abelian surfaces, denoted as $\gG_p=\gG_{p,\ell}$, is the graph whose  vertices set $V$ is the set of  $\overline{\F}_p$-isomorphism classes of $\psas$ and whose edge set $E$ is the set of   equivalence classes of $(\ell,\ell)$-isogenies.
	\end{defn}
	
	Note that two $(\ell,\ell)$-isogenies are equivalent if they have the same kernel. By Theorem~\ref{theorem:3. 13} and Proposition~\ref{prop:lniso}, the number of non-equivalent $(\ell,\ell)$-isogenies from a principally polarized superspecial abelian surface $A$ to another is equal to the number of maximal $\ell$-isotropic subgroups of $A$, which is $(\ell+1)(\ell^2+1)$. Hence 
	\begin{lemma}
		The out-degree of every vertex in $\gG_{p}$  is $(\ell+1)(\ell^2+1)$.
	\end{lemma}

	\begin{defn} Let $\va: E_1\times E_2 \rw A_1$  be an edge in $\mathcal{G}_p$.  
		
		If $\{(P,Q), (P',Q')\}$ is an $\F_{\ell}$-basis of  $\ker(\va)$, we call $\begin{pmatrix}P & P'\\ Q & Q' \end{pmatrix}$ a generator matrix of $\ker(\va)$. The isogeny $\varphi$  is called  diagonal if $\ker(\va)$ has a diagonal generator matrix, and is called  non-diagonal if otherwise.
	\end{defn}
	\begin{lemma} \label{lemma:diagonal}   Let $\va: E_1\times E_2 \rw A_1$  be an edge in $\mathcal{G}_p$. 	For $i=1, 2$, let $K_i =\im(\ker(\va)\hookrightarrow (E_1\times E_2)[\ell]\rw E_i[\ell])$.  Then the followings are equivalent:
		\begin{enumerate}[i]
			\item $\va$ is diagonal.
			\item $\dim K_1=1$.
			\item  $\dim K_2=1$.
			\item There exists some $0\neq P\in E_1[\ell]$ such that $(P,0)\in \ker(\va)$.
			\item There exists some $0\neq Q\in E_2[\ell]$ such that $(0,Q)\in \ker(\va)$. 
		\end{enumerate} 
		In this case $\ker(\va)=K_1\times K_2$.
	\end{lemma}
	\begin{proof} If $\va$ is diagonal, clearly $\dim K_1= \dim K_2=1$ and $\ker(\va)=K_1\times K_2$ by definition.
		
		Suppose $\{(P, Q), (P', Q')\}$ is a basis of $\ker(\va)$ over $\F_\ell$. 
		
		If $\dim K_1=1$, then $P$ and $P'$ are linearly dependent and not both zero.
		We may assume $P'=aP$ and replace  $(P', Q')$ by $(0, Q'-aQ)$, then  $\ker(\va)$ has a basis of the form  $\{(P, Q), (0, Q')\}$ and in particular $Q'\neq 0$. Since  $\ker(\va)$ is a maximal isotropic subgroup, the trivial Weil pairing  $e_{\ell}((P, Q), (0, Q'))= e_{\ell}(P, 0)e_{\ell}(Q, Q')=1$ means  $e(Q, Q')=1$.  Then $Q$ and $Q'$ must be linearly dependent too. Thus $K_2=\la Q'\ra$,  $\ker(\va)=K_1\times K_2$ has a basis of the form $\{(P,0), (0,Q')\}$ and $\va$ is diagonal. 
		
		By symmetry, if $\dim K_2=1$, we also have $\dim K_1=1$, $\ker(\va)=K_1\times K_2$ and $\va$ is diagonal.		Thus (i), (ii), (iii) are equivalent.
		
		Clearly (i) $\Rightarrow$ (iv). On the other hand, we extend $(P,0)$ to a basis $\{(P,0), (P', Q')\}$ of $\ker(\va)$. Then $K_2$ is $1$-dimensional. Hence we have (iv) $\Rightarrow$ (i). Similarly we have (i) $\Leftrightarrow$ (v). 
	\end{proof}
	
	\begin{cor} \label{cor:isogenies} Among isogenies from $E_1\times E_2$ in $\gG_p$, 
		$(\ell+1)^2$ are diagonal and $\ell^3-\ell$ are non-diagonal.
	\end{cor}
	\begin{proof}
		There are $\ell+1$ choices for the $1$-dimensional subspaces $K_1$ of $E_1[\ell]$ and $K_2$ of $E_1[\ell]$ respectively, so there are  $(\ell+1)^2$ diagonal isogenies and $(\ell+1)(\ell^2+1)- (\ell+1)^2=\ell^3-\ell$ non-diagonal ones. 
	\end{proof}

	As in \cite{cds}, we can define the extension of $(\ell, \ell)$- isogenies:
	
	\begin{defn}
		Let $\gA_i=(A, D_i)$ ($i=0, 1, 2$), $\va_1: \gA_0 \rw \gA_1$ and $\va_2: \gA_1 \rw \gA_2$ be edges in $\mathcal{G}_p$.
		\begin{enumerate}[1]
			\item The isogeny $\va_2$ is called a dual extension of $\va_1$ if $\ker(\va_2\circ\va_1)\cong (\Z/\ell\Z)^4$, in this case,  $\ker(\va_2) = \va_1(A[\ell])$.
			\item The isogeny $\va_2$ is called a bad extension of $\va_1$ if $\ker(\va_2\circ\va_1)\cong \Z/\ell^2\Z\times (\Z/\ell\Z)^2$, in this case,  $\ker(\va_2) \cap \va_1(A[\ell])\cong \Z/\ell\Z$.
			\item The isogeny $\va_2$ is called a good extension of $\va_1$ if $\ker(\va_2\circ\va_1)\cong  (\Z/\ell^2\Z)^2$, in this case,  $\ker(\va_2) \cap \va_1(A[\ell])= 0$.
		\end{enumerate}		
	\end{defn}

	\section{Loops and neighbors of $[E_{1728}\times E_{1728}]$}
	
	\subsection{Basic facts}
	In this section, let $E=E_{1728}$ be the supersingular elliptic curve defined over $\F_p$  with $j$-invariant $1728$ (which implies that $p\equiv 3\bmod{4}$). We know its endomorphism ring is
	\[ \gO=\gO_{1728} ={\Z}+{\Z} i+{\Z}\frac{1+j}{2}+{\Z}\frac{i+k}{2}, \]
	where $	i^2=-1,\ j^2=-p,\ ij=-ji=k$.
	Note that the reduced norm on $\gO$ is given by
	\begin{equation} \label{eq:nrd1} \Nrd\left(x+yi+z \frac{1+j}{2}+ w\frac{i+k}{2}\right )= \left(x+\frac{z}{2}\right)^2+ \left(y+\frac{w}{2}\right)^2+\frac{p(z^2+w^2)}{4}. \end{equation}

	Let $[E\times E]=[E_{1728}\times E_{1728}]$ be the superspecial abelian surface $E\times E$ with the principal polarization $\{0\}\times E+ E\times \{0\}$ in the isogeny graph $\gG_p$. 
	
	For $n\in \Q^\times$, let $\sigma(n)=\sigma_1(n)$ be the sum of positive divisors of $n$ if $n\in \Z_+$ and $\sigma(n)=0$ if otherwise.	
	Recall 
	\begin{lem}[Jacobi]\label{dio1}
		For $n\in \Z_+$, the number of integer solutions of  $x^2+y^2+z^2+w^2=n$ is 
		\[ 8\sigma(n)-32\sigma(\frac{n}{4})=8\sum_{d\mid n,\ 4\nmid d} d. \]
		Particularly, there are $24$ integer solutions of $x^2+y^2+z^2+w^2=2$.
	\end{lem}
	
	By simple computation, we have 
	\begin{lemma} \label{lemma:o17281} The group  $G=\Aut(E\times E)=\{ g\in M_2(\gO)\mid g^+g=I\}$ is nothing but the following group of order $32$: 
		\begin{align}
			\left\{ \begin{pmatrix} \pm 1,  \pm i & 0 \\ 0 & \pm 1,  \pm i \end{pmatrix},
			\quad 
			\begin{pmatrix} 0 & \pm 1,  \pm i \\ \pm 1,  \pm i & 0 \end{pmatrix} \right\}	. \label{matrix_M_unitsE1728}
		\end{align} 
	\end{lemma}
	
	We shall need the following notations in this section:
	\begin{enumerate}[i] 
		\item 	If  $\ell\equiv 1\bmod{4}$, let
		\begin{enumerate}
			\item $\lambda:=x_{\ell}+ y_{\ell} i\in \Z[i]$ such that $x^2_{\ell}+ y^2_{\ell}=\ell$;  
			\item $L_1:= \ker(\lambda: E[\ell]\rw  E[\ell])$ and $L_2:= \ker(\bar{\lambda}: E[\ell]\rw  E[\ell])$;
			\item  $R\in L_1\backslash\{0\}$ and $R'\in L_2\backslash\{0\}$;
			\item   $t:=-x_{\ell}/y_{\ell}\in \F_\ell$. Hence $t^2+1=0$.
		\end{enumerate} 
		
		\item	Let $S=R+R'\in E[\ell]$ if $\ell\equiv 1\bmod{4}$ and $S$ be  any fixed  nonzero $P\in E[\ell]$ if $\ell\equiv 3\bmod{4}$. Let $S^*=i(S)$.
		\item Let $G=\Aut(E\times E)$. 
		\item For an isogeny $\va$ starting from $E\times E$, let $G_\va=\{g\in G: \va g=\va\}$ be the stabilizer of $\va$ by the $G$-action, and $O_\va=\{\va g: g\in G\}$ be the $G$-orbit of $\va$.
	\end{enumerate}
	
	Note that $L_1$ and $L_2$ are the only $1$-dimensional invariant $\F_\ell$-subspaces of the operator $i$ on $E[\ell] $, with eigenvalues $t$ and $-t$ respectively.

	\subsection{Kernels of $(\ell,\ell)$-isogenies from  $E_{1728}\times E_{1728}$}

	\begin{lemma} \label{lemma:basis}  The set $\{S, S^*\}$ is an $\F_{\ell}$-basis of   $E[\ell]$. 
	\end{lemma}
	\begin{proof} If $\ell\equiv 3\mod{4}$, we claim that $0\neq P$ and $i(P)$ are linearly independent. Indeed,  if $i(P)= cP$ for some $c\in \F_\ell$, then  $1+ c^2=0$, which is impossible.  
		
		If $\ell\equiv 1\bmod{4}$,  then  $S=R+R'$ and $S^*=tR-tR'$. One can show easily that they are linearly independent. 
	\end{proof}

	\begin{prop}\label{nondiag} \label{prop:diag1} There is a one-to-one correspondence  of the set of non-diagonal $(\ell, \ell)$-isogenies  from $ E\times  E$ and the set of generator matrices 
		\[ \left\{  \begin{pmatrix}
			S & S^*\\ aS+ bS^* & cS+dS^*
		\end{pmatrix}: a, b, c, d\in \F_\ell, ad-bc=-1\right\}. \]
	\end{prop}
	\begin{proof} Since both sets are of order $\ell(\ell^2-1)=\ell^3-\ell$, it suffices to show a non-diagonal isogeny has a generator matrix of the above form.

		Let $\va$ be a non-diagonal $(\ell,\ell)$-isogeny. Suppose $\{(P, Q), (P', Q')\}$ is a basis of $\ker(\va)$ over $\F_\ell$.  
		Then $\{P, P'\}$, $\{Q, Q'\}$ and $\{S, S^*\}$ are all bases of $E[\ell]$ by Lemma~\ref{lemma:diagonal} and Lemma~\ref{lemma:basis}. Suppose $S=a_1 P+ b_1 P'$ and $i(S)=c_1 P + d_1 P'$. 
		Write $a_1 Q+ b_1 Q'=  aS+ b S^*$,  $c_1 Q+d_1 Q'=  cS+ d S^*$.  
		Then $\{(S,aS+ b S^*),  (S^*, cS+ d S^*)\}$ is a new basis of $\ker(\va)$. To ensure $\ker(\va)$ is a maximal isotropic subgroup, we have $e_{\ell}(S,  S^*)e_{\ell}(aS+ b S^*, cS+ d S^*)=1$. This implies $ad- bc+ 1 \equiv 0 \pmod{\ell}$.
	\end{proof}
	\begin{defn} For a  non-diagonal $(\ell, \ell)$-isogeny $\va$  from $ E\times  E$, we call $\{a,b,c,d\}$ given above the quadruple associated to $\va$. 
	\end{defn}
	
	We now describe the action of $G=\Aut(E\times E)$ on the isogenies  explicitly. For the diagonal isogenies, by easy computation, we have
	
	\begin{lemma} Suppose $\ker(\va)=K_1\times K_2$. Then $g\ker(\va)=\ker (\va g^{-1})$ for  each $g\in G$ is given in the following table:
		
		\vskip 0.3cm
		\begin{center}
			\begin{tabular}{|c|c|c|c|}
				\hline
				$g$	& $g\ker(\va)$ & $g$ & $g\ker(\va)$ \\
				\hline
				$\begin{psmallmatrix}\pm 1 & 0\\ 0 & \pm 1 \end{psmallmatrix}$	& $K_1\times K_2$ & $\begin{psmallmatrix}\pm 1 & 0\\ 0 & \pm i \end{psmallmatrix}$ &  $K_1\times i(K_2)$\\
				\hline
				$\begin{psmallmatrix}\pm i & 0\\ 0 & \pm 1 \end{psmallmatrix}$	& $i(K_1)\times K_2$ & $\begin{psmallmatrix}\pm i & 0\\ 0 & \pm i \end{psmallmatrix}$ & $i(K_1)\times i(K_2)$ \\
				\hline
				$\begin{psmallmatrix}0 & \pm 1\\ \pm 1 & 0 \end{psmallmatrix}$	& $K_2\times K_1$ & $\begin{psmallmatrix}0 & \pm 1\\ \pm i & 0 \end{psmallmatrix}$ & $K_2\times i(K_1)$ \\
				\hline
				$\begin{psmallmatrix}0 & \pm i\\ \pm 1 & 0 \end{psmallmatrix}$	& $i(K_2)\times K_1$ & $\begin{psmallmatrix}0 & \pm i\\ \pm i & 0 \end{psmallmatrix}$ & $i(K_1)\times i(K_2)$ \\
				\hline
			\end{tabular}
		\end{center}
	\end{lemma}	
	Consequently, we have
	\begin{prop} \label{diagker}  Diagonal isogenies  can be divided into the following classes, which are unions of $G$-orbits: 
		\begin{enumerate}
			\item[$(\Da)$]  $(K_1, K_2)=(L_1, L_1)$ or $(L_2, L_2)$, where $G_\va= G$ and $O_\va=\{\va\}$.

			\item[$(\Db)$]  $(K_1, K_2)=(L_1, L_2)$ or $(L_2, L_1)$, which form $1$ orbit if  $\ell\equiv 1\bmod{4}$. 
			
			\item[$(\Dc)$] Exactly one of $K_1,K_2$ is in $\{L_1, L_2\}$, where $|G_\va|=8$ and $|O_\va|=4$.  This class contains $4(\ell-1)$ isogenies if $\ell\equiv 1\bmod{4}$.   
			
			\item[$(\Dd)$] $K_1\notin \{L_1, L_2\}$ and $K_2\in \{K_1, i(K_1)\}$, where  $|G_\va|=8$ and $|O_\va|=4$. This class contains  $2(\ell-1)$ isogenies  if $\ell\equiv 1\bmod{4}$, and  $2(\ell+1)$ isogenies if $\ell\equiv 3\bmod{4}$.
			
			\item[$(\De)$] $K_1\notin \{L_1, L_2\}$ and $K_2\notin \{L_1, L_2, K_1, i(K_1)\}$, where  $G_\va= \{\begin{psmallmatrix}	\pm 1 & 0\\ 0 &\pm 1			\end{psmallmatrix}\}$ and $|O_\va|=8$. This class contains  $(\ell-1)(\ell-3)$ isogenies  if $\ell\equiv 1\bmod{4}$, and  $\ell^2-1$ isogenies if $\ell\equiv 3\bmod{4}$.			
		\end{enumerate}
	\end{prop}	
	
	For the non-diagonal isogenies, we have
	
	\begin{lemma} Suppose $\va$ is non-diagonal associated to the quadruple  $\{a,b,c,d\}$. Then for each $g\in G$, the associated quadruple of $\pm \va g^{-1}$  is given in the following table:
		
		\vskip 0.3cm 		
		\begin{center}		\begin{tabular}{|c|c|c|c|c|c|}
				\hline
				$g$ & quadruple & $g$ &  quadruple & $g$ &  quadruple \\
				\hline
				$\begin{psmallmatrix}1 & 0\\ 0 & -1 \end{psmallmatrix}$	& $\{-a,-b,-c,-d\}$ & $\begin{psmallmatrix}i & 0\\ 0 & i \end{psmallmatrix}$ & $\{d,-c,-b,a\}$ & $\begin{psmallmatrix}i & 0\\ 0 & -i \end{psmallmatrix}$ & $\{-d,c,b,-a\}$ \\
				\hline
				$\begin{psmallmatrix}0 & 1\\ -1 & 0\end{psmallmatrix}$ &$\{d,-b,-c,a\}$ &			
				$\begin{psmallmatrix}0 & 1\\ 1 & 0\end{psmallmatrix}$ & $\{-d,b,c,-a\}$	 & 	
				$\begin{psmallmatrix}0 & i\\ i & 0\end{psmallmatrix}$ 	& $\{-a,-c,-b,-d\}$   \\
				\hline
				$\begin{psmallmatrix}0 & i\\ -i & 0 \end{psmallmatrix}$	&  $\{a,c,b,d\}$ & $\begin{psmallmatrix}1 & 0\\ 0 & i \end{psmallmatrix}$ & $\{-b,a,-d,c\}$ & $\begin{psmallmatrix}1 & 0\\ 0 & -i \end{psmallmatrix}$ & $\{b,-a,d,-c\}$ \\
				\hline
				$\begin{psmallmatrix}i & 0\\ 0 & 1 \end{psmallmatrix}$	& $\{-c,-d,a,b\}$ & $\begin{psmallmatrix}i & 0\\ 0 & -1 \end{psmallmatrix}$ & $\{c,d,-a,-b\}$ &  $\begin{psmallmatrix}0 & i\\ 1 & 0 \end{psmallmatrix}$	& $\{-c,a,-d,b\}$   \\
				\hline
				$\begin{psmallmatrix}0 & i\\ -1 & 0 \end{psmallmatrix}$	& $\{c,-a,d,-b\}$ & $\begin{psmallmatrix}0 & 1\\ i & 0 \end{psmallmatrix}$ & $\{-b,-d,a,c\}$ & $\begin{psmallmatrix}0 & 1\\ -i & 0 \end{psmallmatrix}$ & $\{b,d,-a,-c\}$ \\
				\hline
		\end{tabular}	\end{center}
	\end{lemma}
	\begin{proof}  We first show the case for 	$g=\begin{psmallmatrix}i & 0\\ 0 & i \end{psmallmatrix}$. In this case,
		\[ \begin{pmatrix}i & 0\\ 0 & i \end{pmatrix} \begin{pmatrix}
			S & S^*\\ aS+bS^* & cS+dS^*
		\end{pmatrix}= \begin{pmatrix}
			S^* & -S\\ -bS+aS^* & -dS+cS^*
		\end{pmatrix},
		\]
		hence $\{(S, dS-cS^*), (S^*.-bS+aS^*)	\}$ is a basis for 	$\begin{psmallmatrix}i & 0\\ 0 & i \end{psmallmatrix}\va$, thus the associated quadruple is $\{d,-c,-b,a\}$.  
		
		We then show the case for 	$g=\begin{psmallmatrix}0 & i\\ i & 0 \end{psmallmatrix}$. In this case,
		\[ \begin{pmatrix}0 & i\\ i & 0 \end{pmatrix} \begin{pmatrix}
			S & S^*\\ aS+bS^* & cS+dS^*
		\end{pmatrix}= \begin{pmatrix}
			aS^*-bS & cS^*-dS\\ S^* & -S
		\end{pmatrix}.
		\]		
		Note that
		\[ a\begin{pmatrix} 	cS^*-dS\\ -S \end{pmatrix} -  c\begin{pmatrix} 	aS^*-bS\\ S^* \end{pmatrix}  = \begin{pmatrix} 	S\\ -aS-cS^* \end{pmatrix}, \] 
		\[ b\begin{pmatrix} 	cS^*-dS\\ -S \end{pmatrix} -  d\begin{pmatrix} 	aS^*-bS\\ S^* \end{pmatrix}  = \begin{pmatrix} 	S^*\\ -bS-dS^* \end{pmatrix}.  
		\]		
		hence $\{(S, -aS-cS^*), (S^*.-bS-dS^*)	\}$ is a basis for 	$\begin{psmallmatrix}0 & i \\ i & 0 \end{psmallmatrix}\va$, thus the associated quadruple is $\{-a,-c,-b,-d\}$. 		
		
		All other cases are similar.
	\end{proof}
	
	For every $g$, by looking for all $\va$ fixed by $g$ in the above table, we get
	
	\begin{prop}	\label{nondiagker} Non-diagonal isogenies can be divided into the following classes, which are unions of $G$-orbits:
		\begin{enumerate}
			\item[$(\Na)$] $\{a,b,c,d\}=\{a,b,-b,a\}$:
			\begin{enumerate}[i]
				\item $b=\pm a$, where $|G_\va|=8$ and $|O_\va|=4$,
				\item $a=0$ or $b=0$, where $|G_\va|=8$ and $|O_\va|=4$,
				\item other cases of this form, where $G_\va=\{\pm I_2, \pm i I_2\}$ and $|O_\va|=8$.
			\end{enumerate}	
			This class contains $\ell-1$ isogenies if $\ell\equiv 1\bmod{4}$ and  $\ell+1$ isogenies if $\ell\equiv 3\bmod{4}$.
			
			\item[$(\Nb)$] $\{a,b,c,d\}=\{a,b,b,-a\}$, where $|G_\va|=8$  and $|O_\va|=4$. This class contains $\ell-1$ isogenies if $\ell\equiv 1\bmod{4}$ and  $\ell+1$ isogenies if $\ell\equiv 3\bmod{4}$.
			
			\item[$(\Nc)$] $\{a,b,c,d\}=\{a,b,c,-a\}$, $c\neq b$ and $c\neq -b$ if $a=0$, or  $\{a,b,c,d\}=\{a,b,b,d\}$, $d\neq -a$ and $d\neq a$ if $b=0$, where $|G_\va|=4$  and $|O_\va|=8$. This class contains $2(\ell^2-1)$ isogenies.

			\item[$(\Nd)$] all other cases, where $G_\va=\{\pm I_2\} $  and $|O_\va|=16$. This class contains $(\ell-1)(\ell^2-\ell-4)$ isogenies if $\ell\equiv 1\bmod{4}$ and  $\ell(\ell+1)(\ell-3)$ isogenies if $\ell\equiv 3\bmod{4}$.
		\end{enumerate}	
	\end{prop} 
	\begin{proof} We only need to count the number of isogenies in each class. 
		
		First for a fixed $c\in \F_{\ell}^\times$, the equation  $a^2+b^2=c$ has $\ell-1$  pairs of solutions in $\F_\ell$ if $\ell\equiv 1\bmod 4$
		and $\ell+1$ pairs if $\ell\equiv 3\bmod 4$. Take $c=\pm 1$, we get the orders of  Classes $\Na$ and $\Nb$. 
		
		For Class $\Nc$, we need to find the solutions of $a^2+bc=1 (b\neq c, b\neq -c\ \text{if}\ a=0)$ and $b^2-ad=1 (a\neq -d, a\neq d\ \text{if}\ b=0)$. Consider  the first one:
		\begin{enumerate}
			\item if $a=\pm 1$, then we can take either $b=0$ or $c=0$, there are $4\ell-2$ solutions of this type;
			\item if $a\neq \pm 1$, then $1-a^2\neq 0$, there are $\ell-1$ pairs of $b, c$ such that $bc=1-a^2$, so there are $(\ell-2)(\ell-1)$ solutions of this type;
			\item we need to exclude the $b=c$ case, which counts for $\ell-1$ solutions if $\ell\equiv 1\bmod{4}$ and  $\ell+1$ solutions if $\ell\equiv 3\bmod{4}$;
			\item we also need to exclude the $\{0,b,-b,0\}$ case, then $b^2=-1$, which has two solutions if $\ell\equiv 1\bmod{4}$ and  none  if $\ell\equiv 3\bmod{4}$;
			\item in conclusion, it has $\ell^2-1$ solutions. 
		\end{enumerate}
		Similarly the number of solutions for the second one is also $\ell^2-1$. Moreover, these two sets are disjoint, so Class $\Nc$ has $2(\ell^2-1)$ elements.
		
		Class $\Nd$ follows from results for Classes $\Na$-$\Nc$ and Corollary~\ref{cor:isogenies} that the number of non-diagonal isogenies is $\ell^3-\ell$.
	\end{proof}

	\subsection{Loops at $[E_{1728}\times E_{1728}]$}	
	\begin{thm}\label{thmE1728}
		Suppose $p>4\ell$. 
		
		$(1)$ If  $\ell \equiv 1 \bmod 4$, then  the set of loops of $E \times E$ is the union of Classes $\Da, \Db$ and $\Na$.
		
		$(2)$  If $\ell \equiv 3 \bmod 4$, then  the set of loops of $E \times E$ is Class $\Na$.					
		
		$(3)$ If $\ell=2$, then $E \times E$ has $3$ loops.		
	\end{thm}
	\begin{proof} Note that the principal polarized divisor $E\times \{0\}+ \{0\} \times E$  corresponds to the identity matrix $I\in M_2(\gO)$. To determine the loops at $E \times E $ in the isogeny graph $\gG_p$, it is equivalent to determine matrices $M\in M_2(\gO)$ such that $M^{+}M=\ell I$. Moreover, if $M$ corresponds to $\va$, then $\ker\va=\{(P,Q)\in E\times E: M(P,Q)^T=0\}$. Thus  $gM$, where $g\in G$ given in \eqref{matrix_M_unitsE1728},  determines the same loop as $M$ does.

		\vskip 0.3cm		
		\noindent (I) First assume $\ell$ is odd. Write  	
		$$ M= \begin{pmatrix} a & b \\ c & d \end{pmatrix},
		\quad (a, b, c, d \in \gO). $$
		The equation $M^{+}M=\ell I$ implies 
		\begin{align}
			\Nrd(a)+\Nrd(b)=\ell,\quad \Nrd(c)+\Nrd(d)=\ell. \label{1728 ab_ell}
		\end{align}
		Under the condition	$p>4\ell$, by \eqref{eq:nrd1}, similar to the proof of  \cite{OX}, we have
		\[ 		a, b, c, d \in {\Z}[i]. \]
		Replacing $b, c$ by $bi$, $ci$,  we may assume
		\begin{equation} \label{eq:M1728}
			M= \begin{pmatrix} a & bi \\ ci & d \end{pmatrix}
			\quad (a, b, c, d \in {\Z}[i]).\end{equation}
		Now we have 	 $M^{+}M=\ell I= M M^+$. Comparing the coefficients, we get 		
		\begin{align}
			a\bar{a}+ b\bar{b} = c\bar{c}+d\bar{d}= a\bar{a}+c\bar{c} =\ell,\quad   a\bar{c}=b\bar{d}. \label{ac=bd}
		\end{align}
		This implies
		\begin{equation}
			a\bar{a}=d\bar{d},\quad  b\bar{b} = c\bar{c},\ a\bar{c}=b\bar{d}. \label{nrda=nrdd}
		\end{equation}
		
		\noindent (i) If one of $a$,$b$ is $0$, up to an element $g\in G$, we may assume $b=0$, then $c=0$ and $a\bar{a}=d\bar{d}=\ell$, which can only happen if $\ell\equiv 1\bmod{4}$. In this case,  $M$ is one of the following matrices 
		\[  \begin{pmatrix} \lambda & 0 \\ 0 & \lambda \end{pmatrix},\ 
		\begin{pmatrix} \lambda & 0 \\ 0 & \bar{\lambda} \end{pmatrix},\ 
		\begin{pmatrix} \bar{\lambda} & 0 \\ 0 & \lambda \end{pmatrix},\ 
		\begin{pmatrix} \bar{\lambda} & 0 \\ 0 & \bar{\lambda} \end{pmatrix},
		\]
		which correspond to the $4$ isogenies in Classes $\Da$ and $\Db$.
		
		\noindent (ii) If $ab\neq 0$, then $c,d$ are also not $0$. Note that ${\Z}[i]$ is a principal ideal domain.  Let $A=\gcd(a,d)$ and $B=a/A$.
		Let $C=\gcd(b,c)$ and $D=b/C$.  
		Then \eqref{nrda=nrdd} implies that		
		\[		a=AB,\ d=A\bar{B},\ b=CD,\ c=C\bar{D}.	\]
		The equation $a\bar{c}=b\bar{d}$ implies that  $A\bar{C}=\bar{A}C$ and hence $A\bar{C} \in {\Z}$. 
		
		If $\pi$ is a Gauss prime above a rational prime $q$ such that the $\pi$-adic valuation of  $A$ is larger than the $\bar{\pi}$-adic valuation of $A$, then  $\pi\mid C$, hence $\pi$ is a common divisor of $a,b,c,d$ and then $q\mid a\bar{a}+b\bar{b}=\ell$, which means $q=\ell$ and one of $a, b$ must be $0$, which is a contradiction to the condition $ab\neq 0$. Hence $A=\varepsilon A'$ for some $A'\in \Z$ and $\varepsilon=\pm 1,\pm i$, so  $C=\varepsilon C'$ for some $C'\in \Z$.

		In this case, up to $g\in G$ we may assume  $a=\bar{d}$, $b=\bar{c}$, then  
		\begin{equation} \label{eq:loopmatrix} M=\begin{pmatrix} a & bi \\ \bar{b}i & \bar{a} \end{pmatrix}. \end{equation}
		Write $a=a_1+a_2i$, $b=b_1+b_2i$. The equation  $a\bar{a}+ b\bar{b}=\ell$ implies 
		\[ {a_1}^2+{a_2}^2+{b_1}^2+{b_2}^2=\ell, (a_1, a_2)\neq (0,0),\  (b_1, b_2)\neq (0,0). \] 
		By  Lemma \ref{dio1}, the  equation above has $8(\ell+1)$ solutions if $\ell\equiv 3\bmod{4}$ and  $8(\ell-1)$ solutions if $\ell\equiv 1\bmod{4}$. Now if $M$ is of the form \eqref{eq:loopmatrix}, then  $gM$ is also of this form  exactly  when
		\[ g\in \left\{\pm I_2,\ \pm \begin{pmatrix}
			i  & 0\\ 0 & -i
		\end{pmatrix},\ \pm \begin{pmatrix}
			0  & 1\\ -1 & 0
		\end{pmatrix},\ \pm \begin{pmatrix} 
			0  & i\\ i & 0
		\end{pmatrix} \right\}.\]
		Thus the number of  non-diagonal loops  is  $\ell+1$  if $\ell\equiv 3\bmod{4}$ and  $\ell-1$ if $\ell\equiv 1\bmod{4}$.
		
		Finally let us compute the kernels in this case. Suppose $M$ corresponds to the loop $\varphi$, then  $(P,Q)\in \ker\va$ means $M (P,Q)^T=0$. If $0\neq P\in E[\ell]$, then $aP\neq 0$, otherwise $a\bar{a} P=0$ and hence $P=0$ as $\ell\nmid a\bar{a}$. So $(P,0)\notin \ker\va$. Thus $\ker\va$ is non-diagonal. By Proposition~\ref{prop:diag1}, suppose $\ker\va$ is associated with $\{u,v,w,t\}$  with $ut-vw=-1$. 
		Then 
		\[ 
		\begin{pmatrix}
			a& bi\\ \bar{b}i & \bar{a} 
		\end{pmatrix} \begin{pmatrix}
			S & S^*\\
			uS+ v S^* & wS+ t S^*
		\end{pmatrix}= \begin{pmatrix}
			0 & 0 \\	0& 0
		\end{pmatrix}. \]
		Hence $(a+ ub i- vb)(S)= (ai+ wbi- tb)(S)= 0$. This implies $b(-u- vi- wi+ t)(S)=0$. Again by $\ell\nmid b\bar{b}$, we have  $(-u- vi- wi+ t)(S)=0$. Since $\{S, i(S)=S^*\}$ is a basis of $E[\ell]$, we have $t=u$ and  $w=-v$ in $\F_{\ell}$. In fact, 
		\[ (u,v)=\left(\frac{a_1 b_2 -a_2b_1}{b_1^2+b_2^2}, \frac{a_1 b_1+a_2 b_2}{b_1^2+b_2^2} \right). \]
		Thus the loops here are exactly those contained in Class $\Na$.
		
		\vskip 0.3cm
		\noindent (II) Assume  $\ell=2$.  By  calculation there are $3$ loops which correspond to the three matrices 
		\[\begin{pmatrix} 1 & 1 \\ -1 & 1 \end{pmatrix},
		\quad
		\begin{pmatrix} 1 & i \\ i & 1 \end{pmatrix},
		\quad  
		\begin{pmatrix} 1+i & 0 \\ 0 & 1+i \end{pmatrix}. \qedhere \]
	\end{proof}

	\subsection{Neighbors of $[E_{1728}\times E_{1728}] $} 
	
	If $\ell\equiv 1\bmod{4}$,	we denote the loops  $\lambda\times \lambda$ and $\bar{\lambda}\times \bar{\lambda}$ on $E\times E $ by   $[\lambda]$ and $[\bar{\lambda}]$ respectively. Hence
	\[ \ker[\lambda]=L_1\times L_1,\quad \ker[\bar{\lambda}]=L_2\times L_2. \]
	
	\begin{lemma}\label{goodiso} Suppose $\ell\equiv 1\bmod{4}$ and $p>4\ell$. 
		For a loop $\alpha\neq [\lambda]$, $[\bar{\lambda}]$ of  $E\times E$ of degree $\ell^2$, let
		\[ L_{i}^\al=(L_i\times L_i)\cap \ker\al\ \ (i=1, 2). \]  
		Then $\alpha\mapsto L_i^{\alpha}$ gives a one-to-one correspondence 
		\[ \{\text{degree $\ell^2$ loops}\ \neq[\lambda]\ \text{or}\ [\bar{\lambda}]\}\leftrightarrow \{\text{$1$-dimensional subspaces of } L_i\times L_i\}.\]
		Consequently, $\ker\al=L_1^\al\times L_2^\al$, and if $L_1^\al=\{(P, kP): P\in L_1\}$ for some $k\neq 0$, then $L_2^\al=\{(-kQ, Q): Q\in L_2\}$.   
	\end{lemma}
	\begin{proof} 
		We prove the one-to-one correspondence for $\al\leftrightarrow L_1^\al$. The case for  $\al\leftrightarrow L_2^\al$ is parallel.
		
		If $\ker\al=L_1\times L_2$ or $L_2\times L_1$, clearly $L_1^\al=L_1\times \{0\}$ or $\{0\}\times L_1$ which is $1$-dimensional. Conversely these two lines correspond to these two diagonal kernels.
		
		If $\al$ is non-diagonal, suppose it is associated to the quadruple $(u,v,-v,u)$ such that $u^2+v^2=-1$, then $(R, (u+vt) R)\in \ker\al\cap (L_1\times L_1)$ which must be one dimensional.
		
		On the other hand, we may assume the $1$-dimensional subspace is generated by $(R, kR)$ for some $k\neq 0$.  By  Theorem~\ref{thmE1728}, we need to find a unique quadruple $(u,v,-v,u)$ such that $u^2+v^2=-1$ and $(R, kR)\in \la (S, uS+vS^*), (S^*, -vS+uS^*) \ra$. 
		Note that $iR=tR$, $iR'=-tR'$, $S=R+R'$ and $S^*=tR-tR'$, where  $t^2=-1$. 	Then
		\[ \begin{pmatrix} R \\ uR+vtR 	
		\end{pmatrix}=\frac{1}{2} \begin{pmatrix} S \\ uS+v S^* 	
		\end{pmatrix}-\frac{t}{2} \begin{pmatrix} S^* \\ -vS+ uS^* 	
		\end{pmatrix}.
		\]
		By Lemma~\ref{lemma:diagonal}, a non-diagonal kernel can not contain two elements of the form $(R, *)$, hence $u+vt=k$. Then $u^2+v^2=(u+vt)(u-vt)=-1$ and hence $u-vt=-k^{-1}$. There is a unique $(u,v)=(\frac{k-k^{-1}}{2}, \frac{k+k^{-1}}{2t})$ satisfying these conditions.  We also see $L_2^\al$ in this case is the $1$-dimensional subspace generated by $(R', -k^{-1}R')$ or equivalently by $(-kR', R')$. 
	\end{proof}

	\begin{prop}\label{llc}
		Suppose $p>4\ell^2$. Every loop $\va$ of $E\times E$ of degree $\ell^4$ corresponds to a matrix of the following form:
		\begin{enumerate}[I]
			\item $\ell I_2$. In this case $\ker\va\cong (\Z/\ell\Z)^4$.  
			\item $\begin{pmatrix}
				a & bi \\
				\bar{b}i  & \bar{a} 
			\end{pmatrix}$, where $a, b\in \Z[i]$, $a\bar{a}+ b\bar{b}= \ell^2$ and $ab\neq 0$.   In this case $\ker(\va)\cong (\Z/\ell^2\Z)^2$.
			\item $\lambda \begin{pmatrix}
				a & bi\\
				\bar{b}i & \bar{a}
			\end{pmatrix}$ or $ \bar{\lambda} \begin{pmatrix}
				a & bi\\
				\bar{b}i & \bar{a}
			\end{pmatrix}$, where $a, b\in \Z[i]$ and $a\bar{a}+ b\bar{b}= \ell$.  In this case $\ker\va\cong (\Z/\ell\Z)^2\times \Z/\ell^2\Z$ which occurs only if $\ell\equiv 1\bmod{4}$.
		\end{enumerate}
		
		As a consequence, every loop of degree $\ell^4$ can be factorized as the product of two loops of degree $\ell^2$. Moreover, loops in Case $(2)$ is uniquely factorized as the composition of two edges of degree $\ell^2$. 
	\end{prop}
	
	\begin{proof} Suppose $M$ corresponds to $\va$. Then $M^+M=\ell^4 I$. If $p>4\ell^2$, by following the same argument in the proof of Theorem~\ref{thmE1728}, we can deduce that $M\in M_2(\Z[i])$ has the form of  I, II or III. 
		
		Clearly a loop $\va$ in Case I is the composition of $\al$ and $\hat{\al}$ where $\al$ is a loop of degree $\ell^2$. It is also clear $\ker(\al)\cong (\Z/\ell\Z)^4$. 
		
		A loop in Case III is the extension of $[\lambda]$ or $[\bar{\lambda}]$ of a non-diagonal loop $\al$ of degree $\ell^2$. By Lemma~\ref{goodiso}, $L_1^\al\cong L_2^\al \cong \Z/\ell\Z$, thus this  extension is bad   and $\ker\va=\Z/\ell^2\Z\times (\Z/\ell\Z)^2$. 
		
		Now let $\va= \al\circ \beta$ be a loop of degree $\ell^4$ where $\al, \beta$ are two loops of degree $\ell^2$ such that $\al, \beta \notin \{[\lambda],\ [\bar{\lambda}]\}$ and $\beta\ne \hat{\al}$.  
		
		We first claim that $\al$ is a good extension of $\beta$. If not, then $\al$ is a bad extension of $\beta$. Hence $\beta((E\times E)[\ell]) \cap \ker(\al)\cong \Z/\ell\Z$ is an $\F_{\ell}$-line generated   by some  $(P, Q) \in (E\times E)[\ell]$. Since $i\circ \beta= \beta\circ i$, we have
		\[ (i(P), i(Q)) \in i\circ \beta((E\times E)[\ell])= \beta\circ i((E\times E)[\ell])= \beta((E\times E)[\ell]). \]
		Similarly, by $i\circ \al= \al\circ i$, we have $(i(P), i(Q)) \in \ker(\al)$. Hence $(i(P), i(Q))= c (P, Q)$ for some $c\in \F_\ell$. This is impossible if $\ell\equiv 3\bmod{4}$. If $\ell\equiv 1\bmod{4}$, then  both of $P$ and $Q$ are in either $L_1$ or $L_2$. Since $\beta((E\times E)[\ell]) \subseteq \ker(\hat{\beta})$,  we have $(P, Q)\in \ker(\hat{\beta})\cap \ker(\al)$. However, by Lemma~\ref{goodiso}, there is only one loop of degree $\ell^2$ not of the form $[\lambda]$ whose kernel contains $(P,Q)$, thus $\hat{\beta}=\al$, which is a contradiction.
		
		We then show  the factorization of $\va$ as $\al\circ \beta$, where $\al$ and $\beta$ are two edges of degree $\ell^2$, is unique. Indeed,  
		$\ker(\va)\cong (\Z/\ell^2\Z)^2$ has a unique subgroup which is isomorphic to $(\Z/\ell\Z)^2$. This subgroup must be  $\ker(\beta)$. Thus $\beta$ is unique and so is $\al$.  
		
		In conclusion, loops which is the composition of two loops $\neq [\lambda],\ [\bar{\lambda}]$ and without backtracking are all in Case II, with kernels $\cong (\Z/\ell^2\Z)^2$, and the number of such loops is $(\ell+1)\ell$. 
		However, following the proof of Theorem~\ref{thmE1728}, the number of loops of degree $\ell^4$ in Case II is $\sigma(\ell^2)- 1= \ell^2+\ell$. Thus loops of degree $\ell^4$ in Case II are all products of loops of degree $\ell^2$.
	\end{proof}

	\begin{lemma}\label{llc2}
		Suppose  $p> 4\ell^2$ and $\ell \equiv 1 \pmod 4$.
		
		If  $\al$ and $\beta$ are two edges from $E\times E$ to a vertex $V\neq E\times E$ such that  the loop $\va=\hat{\al}\circ \beta$ has a factor of $[\lambda]$ or $[\bar{\lambda}]$. Then
		\begin{enumerate}[1]
			\item $\al$ and $\beta$ are not in the same $G$-orbit.
			\item If $\va=\tau\circ [\lambda]$, then $\hat{\va}=\hat{\beta}\circ \al= \hat{\tau}\circ [\bar{\lambda}]$.
			\item If $\va=\tau\circ  [\lambda]$, then
			\begin{equation} \label{eq:kerg1} \ker(\beta)\cap (L_1\times L_1)= L_1^\tau,\quad  \ker(\al)\cap (L_2\times L_2)= L_2^{\hat{\tau}}. \end{equation}
			Hence $\tau$, $\al$ (resp. $\tau$, $\beta$) and  $\va$ are uniquely determined by $\beta$ (resp. $\al$).	
		\end{enumerate}
		
		On the other hand, if an edge $\beta: E\times E\rightarrow V$ of degree $\ell^2$ satisfies
		\begin{equation} \label{eq:kerg} \dim \ker(\beta)\cap (L_i\times L_i)=1, \end{equation}
		for $i=1$ (resp. $i=2$), then there exists $\al:  E_{1728}\times E_{1728}\rightarrow V$ such that $\va=\hat{\al}\circ \beta=\tau\circ [\lambda]$  (resp. $\va=\hat{\al}\circ \beta=\tau\circ [\bar{\lambda}]$)  for some loop $\tau$.
	\end{lemma}		
	\begin{proof} (1) is easy, since if $\al=\beta g$ for $g\in G$, then $\hat{\beta}\circ \al=\ell g$ belongs to Case (1) in Proposition~\ref{llc}. (2) follows from the fact that $[\lambda]$ commutes with $\tau$. 
		
		For (3), we only need to study $\ker(\beta)$. Note that
		\begin{enumerate}[a]
			\item  $\va$ belongs to Case (3) in Proposition~\ref{llc}, thus $\ker(\va)\cong (\Z/\ell\Z)^2\times \Z/\ell^2\Z$;
			\item $L_2^\tau= (L_2\times L_2)\cap \ker(\tau)\subseteq \ker(\va)$; 
			\item  $L_1\times L_1=\ker [\lambda]\subseteq \ker(\va)$. 
		\end{enumerate}
		Thus the $\ell$-part of $\ker(\va)$, which contains $\ker(\beta)$, is generated by $L_1\times\{0\}=\la (R, 0)\ra$, $\{0\}\times L_1=\la (0, R)\ra$ and $L_2^\tau$. 
		\begin{enumerate}[i]
			\item If $L_2^\tau=L_2\times \{0\}=\la (R', 0)\ra $, then $L_1^\tau=\{0\}\times L_1$. By computation, $\ker(\beta)= \la (U, 0), (0, R)\ra$ where $U\in E_{1728}[\ell]$, $U\notin L_1\cup L_2$. Thus
			\[ \ker(\beta)\cap (L_1\times L_1)= L_1^\tau,\quad \ker(\beta)\cap (L_2\times L_2)=0. \] 
			
			\item If $L_2^\tau=\{0\}\times L_2=\la (0,R')\ra $, then $L_1^\tau=L_1\times \{0\}$. The proof is similar to (i).
			
			\item 	If 	$L_2^\tau=\la(R', aR')\ra$ for some $a\neq 0$, by Lemma~\ref{goodiso}, then $L_1^\tau=\la(R, -a^{-1}R)\ra$.  In this case  $\la (R, 0), (0, R), (R', aR')\ra\supset \ker(\beta)$.
			We assume $\ker(\beta)= \la (R, cR), (R', bR+ aR')\ra$. By computing the Weil pairing, we get $c=-a^{-1}$ and $\ker(\beta)\cap (L_1\times L_1)= L_1^\tau$ and $\ker(\beta)\cap (L_2\times L_2)=0$. 
		\end{enumerate}
		This finishes the proof of (3). 
		
		Now suppose $\beta$ satisfies \eqref{eq:kerg} for $i=1$, we let $\tau$ be the loop such that $L_1^\tau=L$. Checking the argument above we see that  $\ker(\beta)\subset \ker( \tau\circ [\lambda])$, so $\va= \tau\circ [\lambda]$ factors through $\beta$.  
	\end{proof}
	\begin{lemma} \label{lemma:kerg} Suppose $\ell\equiv 1\bmod{4}$ and $p>4\ell^2$. 
		\begin{enumerate}[1]
			\item The diagonal non-loop isogenies satisfying \eqref{eq:kerg} are exactly those in Class $\Dc$, whose number is $4(\ell-1)$.
			\item The quadruples $\{a,b,c,d\}$ associated to  non-diagonal non-loop isogenies  satisfying \eqref{eq:kerg} for $i=1$ are parameterized by  
			\begin{equation}  a= k-k^{-1} -d,\ b= dt-kt,\ c=dt+ k^{-1} t, \end{equation}
			where $k\in \F_{\ell}^\times, d\in \F_{\ell}, k- k^{-1}\neq 2d$. If $k=\pm 1$ and $d\neq 0$ or $d\neq k=\pm t$, then the isogenies belong to Class $\Nc$ , other cases belong to Class $\Nd$. 
			The number of isogenies satisfying \eqref{eq:kerg}  for $i=1,2$ is $2(\ell-1)^2$, of which $8(\ell-1)$ are in Class $\Nc$ and $2(\ell-1)(\ell-5)$  in Class $\Nd$.
		\end{enumerate}
	\end{lemma}	
	\begin{proof} (1) is easy. For (2), if $\ker(\beta)$ contains an element $(R,uR)$ for some $u\neq 0$, then $(R,uR)=\frac{1}{2}[(S, aS+bS^*)-t(S^*, cS+dS^*) ]$, which means
		\[ ad=bc-1,\ a-d= bt+ct.\]
		Hence $(a+d)^2= (a-d)^2+4ad=-(b+c)^2+4bc-4=-(b-c)^2-4$, and 
		\[ (a+d)^2+(b-c)^2= (a+d+bt-ct)(a+d-bt+ct)=-4. \] 
		Plug in $a=d+bt+ct$, we get
		\[ d+bt=k,\ d+ct=-k^{-1}.  \]
		This gives the parametrization. However, we need to exclude the loop case, which means $a=d$ or equivalently $k- k^{-1}=2d$.
	\end{proof}
	
	\begin{thm} \label{thm:trans} Suppose $p>4\ell^2$.  Let $V$ be a vertex adjacent to $E\times E$.
		\begin{enumerate}[1]
			\item If there exists one isogeny $\beta: E\times E\rightarrow V$ satisfying \eqref{eq:kerg}, then all isogenies from $E\times E$ to $V$ satisfy \eqref{eq:kerg} and they form  two $G$-orbits.  
			\item For all other cases, $(\ell,\ell)$-isogenies from  $E\times E$ to $V$ form a $G$-orbit. 
		\end{enumerate}
	\end{thm}
	\begin{proof} Suppose $\al, \beta$ are two edges from $E\times E$ to $V$, then $\va=\hat{\al}\circ \beta$ is a loop of degree $\ell^4$. Moreover, by the unique factorization of loops of Case II in Proposition~\ref{llc},  $\va$ doesn't belong to this case. By Lemma~\ref{llc2}, $\va$ belongs to Case III only if $\beta$ satisfies \eqref{eq:kerg}.
		For all other cases, $\va$ must be in Case I and $\al$ and $\beta$ are in the same $G$-orbit.
		
		Now suppose $\va$ belongs to Case III, then by 	Lemma~\ref{llc2}, $\al$ and $\beta$ are in two different $G$-orbits. Moreover $\hat{ g_1}\beta g_2$ for all $g_1, g_2\in G$ are in Case III, which in turn means that $\al g_1$ and $\beta g_2$ all satisfy \eqref{eq:kerg}. 
		
		We only need to show there is no other $G$-orbit. If not, suppose $\al'$ is an edge from  $E\times E$ to $V$ not in $\al G$ and $\beta G$. Then both $\va=\hat{\al}\circ\beta$ and $\va'= \hat{\al'}\circ\beta$ factor through $[\lambda]$ or $[\bar{\lambda}]$, however, there is only one such $\va$ by Lemma~\ref{llc2}.
	\end{proof}

	\begin{thm}\label{main1}
		Suppose $p>4\ell^2$. Consider the neighbor of $[E_{1728}\times E_{1728}]$.
		
		$(1)$ If  $\ell \equiv 1 \bmod 4$, the neighbor is given by the following   table:
		
		\begin{center}
			\begin{tabular}{|c|c|c||c|c|c|}
				\hline
				\#{Vertices} &  Multi-Edges & Edge Type	& \#{Vertices} & Multi-Edges & Edge Type	\\
				\hline
				$\frac{\ell-1}{2}$ & 8 &$\Dc$	&  $\frac{(\ell-1)(\ell-3)}{4}$ & 8 & $\Nc$-1 \\
				\hline
				$\frac{\ell-1}{2}$ & 4 &$\Dd$	&  $\frac{\ell-1}{2}$ & 16 &$\Nc$-2   \\
				\hline
				$\frac{(\ell-1)(\ell-3)}{8}$ & 8 &$\De$	&  $\frac{(\ell-1)(\ell^2-3\ell+6)}{16}$ & 16 &$\Nd$-1  \\
				\hline
				$\frac{\ell-1}{4}$ & 4 &$\Nb$	 & $\frac{(\ell-1)(\ell-5)}{16}$ & 32 & $\Nd$-2   \\
				\hline
			\end{tabular}
		\end{center}

		$(2)$ If $\ell \equiv 3 \bmod 4$, the neighbor is given by the following   table:
		
		\begin{center}
			\begin{tabular}{|c|c|c||c|c|c|}
				\hline
				\#{Vertices} & Multi-Edges &  Edge Type & \#{Vertices} & Multi-Edges & Edge Type	\\
				\hline
				$\frac{\ell+1}{2}$ & 4 & $\Dd$	 & $\frac{\ell+1}{4}$ & 4 &$\Nb$\\
				\hline
				$\frac{\ell^2-1}{8}$ & 8 &	$\De$	 &  $\frac{\ell^2-1}{4}$  &  8 & $\Nc$ \\
				\hline
				&  &   &  $\frac{\ell(\ell+1)(\ell-3)}{16}$  & 16 & $\Nd$ \\
				\hline
			\end{tabular}
		\end{center}

		$(3)$ If $\ell=2$, there are $3$ vertices adjacent
		to $[E_{1728}\times E_{1728}]$ ,  each connecting  with $4$ edges, $2$ vertices with diagonal and $1$ with non-diagonal kernels.	
	\end{thm}
	\begin{proof} For $\ell$ odd, this is a consequence of
		Proposition~\ref{diagker}, Proposition~\ref{nondiagker}, Lemma~\ref{lemma:kerg} and Theorem~\ref{thm:trans}.

		Now we consider the case  $\ell= 2$. Suppose   $P\in E[2]$ such that $P\ne i(P)$. Then  $E[2]=\{O,P, Q=i(P), S=P+Q\}$. Moreover $i(S)= S$. 
		Let $K_1=\{O,P\}$, $K_2=\{O,Q\}$ and $L=\{O,S\}$. Then the only diagonal  loop kernel  is $L\times L$. The 
		other $8$ diagonal isogenies form two $G$-orbits: 
		\[\{K_1\times K_1, K_1\times K_2, K_2\times K_1, K_2\times K_2\},\ 
		\{L\times K_1, L\times K_2, K_1\times L, K_2\times L\}. \]
		The kernels of the $2$    non-diagonal loops  are $\la (P, P), (Q, Q)\ra$ and $ \la(P, Q), (Q, P)\ra$. The other $4$  non-diagonal isogenies belong to one orbit.
	\end{proof}
	
	\section{Loops and neighbors of $[E_{0}\times E_{0}]$}
	In this section, let $E_{0}$ be the supersingular elliptic curve defined over $\F_p$  with $j$-invariant $0$ (which implies that $p\equiv 2\bmod{3}$). We know its endomorphism ring is
	\[ \gO_{0} ={\Z}+{\Z}\frac{1+i}{2}+{\Z}\frac{i+k}{3}+{\Z}\frac{j+k}{2},\ (i^2=-3,\ j^2=-p,\ ij=-ji=k). \]
	Note that the reduced norm is given by 
	\begin{equation}  \label{eq:nrd2} \begin{split}
			&\Nrd\left(x+y\frac{1+i}{2}+z \frac{i+k}{3} + w\frac{j+k}{2}\right )\\ &= \left(x+\frac{y}{2}\right)^2+ 3\left(\frac{y}{2}+\frac{z}{3}\right)^2+\frac{p(z^2+3zw+3w^2)}{3}.
	\end{split} \end{equation}

	Let $[E_{0}\times E_{0}]$ be the superspecial abelian surface $E_{0}\times E_{0}$ with the principal polarization $\{0\}\times E_{0}+ E_{0}\times \{0\}$ in the isogeny graph $\gG_p$. 
	
	Now results hereafter in this section are parallel to those in the previous section, whose proofs are almost identical and will be omitted.
	
	\begin{lem}\label{dio2}
		The number of integer solutions of Diophantine equation  $$x^2+xy+y^2+z^2+zw+w^2=n$$ is $12\sigma(n)-36\sigma(\frac{n}{3})$.
	\end{lem}

	\begin{lemma} \label{lemma:o01} The group  $G=\Aut(E_0\times E_0)=\{ g\in M_2(\gO_0)\mid g^+g=I\}$ is the following group of order $72$: 
		\begin{align}
			\left\{	\begin{pmatrix} \pm 1,  \pm \frac{1+i}{2},  \pm \frac{1-i}{2} & 0 \\ 0 & \pm 1,  \pm \frac{1+i}{2}, \pm\frac{1-i}{2} \end{pmatrix}, \quad
			\begin{pmatrix} 0 & \pm 1,  \pm \frac{1+i}{2},  \pm \frac{1-i}{2} \\ \pm 1,  \pm \frac{1+i}{2},  \pm \frac{1-i}{2} & 0 \end{pmatrix}\right\}.\label{matrix_M_unitsE0}
		\end{align} 
	\end{lemma}
	
	We shall need the following notation in this section: 
	
	\begin{enumerate}[i] 
		\item 	If  $\ell\equiv 1\bmod{3}$, let
		\begin{enumerate}
			\item $\lambda:=x_{\ell}+ y_{\ell} \frac{1+ i}{2}\in \Z[\frac{1+ i}{2}]$ such that $x^2_{\ell}+ x_{\ell}y_{\ell}+ y^2_{\ell}=\ell$;  
			\item $L_1:= \ker(\lambda: E_{0}[\ell]\rw  E_{0}[\ell])$ and $L_2:= \ker(\bar{\lambda}: E_{0}[\ell]\rw  E_{0}[\ell])$; 
			\item  $R\in L_1\backslash\{0\}$ and $R'\in L_2\backslash\{0\}$;
			\item   $t:=-x_{\ell}/y_{\ell}\in \F_\ell$. Note that $t^2- t+1=0$.
		\end{enumerate} 
		
		\item	Let $S=R+R'\in E_{0}[\ell]$ if $\ell\equiv 1\bmod{3}$ and $S=P$ be a fixed  nonzero $P\in E_{0}[\ell]$ if $\ell\equiv 2\bmod{3}$. Let $S^*=\frac{1+ i}{2}(S)$.
		\item Let $G=\Aut(E_{0}\times E_{0})$. 
		\item For an isogeny $\va$ starting from $E_{0}\times E_{0}$, let $G_\va=\{g\in G: \va g=\va\}$ be the stabilizer of $\va$ by the $G$-action, and $O_\va=\{\va g: g\in G\}$ be the $G$-orbit of $\va$.
	\end{enumerate}
	
	Note that $L_1$ and $L_2$ are the only $1$-dimensional invariant $\F_\ell$-subspaces of the operator $\frac{1+ i}{2}$ on $E_{0}[\ell] $, with eigenvalues $t$ and $-t$ respectively; similarly,  $L_1$ and $L_2$ are the only $1$-dimensional invariant $\F_\ell$-subspaces of the operator $\frac{1- i}{2}$ on $E_{0}[\ell] $, with eigenvalues $1- t$ and $1+ t$ respectively.

	\subsection{Kernels of $(\ell,\ell)$-isogenies from  $E_{0}\times E_{0}$}

	\begin{lemma} \label{lemma:basis1}  The set $\{S, S^*\}$ is an $\F_{\ell}$-basis of   $E_{0}[\ell]$. 
	\end{lemma}

	\begin{prop}\label{nondiag1} \label{prop:diag11} There is a one-to-one correspondence  of the set of non-diagonal $(\ell, \ell)$-isogenies  from $ E_{0}\times  E_{0}$ and the set of generator matrices 
		\[ \left\{  \begin{pmatrix}
			S & S^*\\ aS+ bS^* & cS+dS^*
		\end{pmatrix}: a, b, c, d\in \F_\ell, ad-bc=-1\right\}. \]
	\end{prop}

	\begin{defn} For a  non-diagonal $(\ell, \ell)$-isogeny $\va$  from $ E_{0}\times  E_{0}$, we call $\{a,b,c,d\}$ given above the quadruple associated to $\va$. 
	\end{defn}
	
	We now describe the action of $G=\Aut(E_{0}\times E_{0})$ on the isogenies  explicitly. For the diagonal isogenies, we have
	
	\begin{prop} \label{diagker}  Suppose $\ker(\va)=K_1\times K_2$. Diagonal isogenies  can be divided into the following classes, which are unions of $G$-orbits: 
		\begin{enumerate}
			\item[$(\Da)$]  $(K_1, K_2)=(L_1, L_1)$ or $(L_2, L_2)$, where $G_\va= G$ and $O_\va=\{\va\}$.

			\item[$(\Db)$]  $(K_1, K_2)=(L_1, L_2)$ or $(L_2, L_1)$, which form $1$ orbit if  $\ell\equiv 1\bmod{3}$. 
			
			\item[$(\Dc)$] Exactly one of $K_1,K_2$ is in $\{L_1, L_2\}$, where $|G_\va|=12$ and $|O_\va|=6$.  This class contains $4(\ell-1)$ isogenies if $\ell\equiv 1\bmod{3}$.   
			
			\item[$(\Dd)$] $K_1\notin \{L_1, L_2\}$ and $K_2\in \{K_1, \frac{1+ i}{2}(K_1), \frac{1- i}{2}(K_1)\}$, where  $|G_\va|=8$ and $|O_\va|=9$. This class contains  $3(\ell-1)$ isogenies  if $\ell\equiv 1\bmod{3}$, and  $3(\ell+1)$ isogenies if $\ell\equiv 2\bmod{3}$.
			
			\item[$(\De)$] $K_1\notin \{L_1, L_2\}$ and $K_2\notin \{L_1, L_2, K_1, \frac{1+i}{2}(K_1), \frac{1- i}{2}(K_1)\}$, where $G_\va= \{\begin{psmallmatrix}	\pm 1 & 0\\ 0 &\pm 1			\end{psmallmatrix}\}$ and $|O_\va|=18$. This class contains  $(\ell-1)(\ell-4)$ isogenies  if $\ell\equiv 1\bmod{3}$, and  $\ell^2-\ell-2$ isogenies if $\ell\equiv 2\bmod{3}$.			
		\end{enumerate}
	\end{prop}	
	
	For the non-diagonal isogenies, we have

	\begin{prop}	\label{nondiagker1} Non-diagonal isogenies can be divided into the following classes, which are unions of $G$-orbits:
		\begin{enumerate} 
			\item[$(\Na)$] $\{a,b,c,d\}=\{a,b,-b,a+ b\}$:
			This class contains $\ell- 1$ isogenies if $\ell\equiv 1\bmod{3}$ and  $\ell+1$ isogenies if $\ell\equiv 2\bmod{3}$.
			
			\item[$(\Nb)$] $\{a,b,c,d\}=\{a,b, a+b, -a\}$, where $|G_\va|=12$  and $|O_\va|=6$. This class contains $\ell-1$ isogenies if $\ell\equiv 1\bmod{3}$ and  $\ell+1$ isogenies if $\ell\equiv 2\bmod{3}$.
			
			\item[$(\Nc)$] if $3d^2\ne -1$, $\{a,b,c,d\}=\{a,b,c,-a\}$, $c\neq a+ b$, or  $\{a,b,c,d\}=\{a,b,a+ b,d\}$, $d\neq -a$, or $\{a,b,c,d\}=\{a,b,c,b- c\}$, $d\neq -a$ where $|G_\va|=4$  and $|O_\va|=18$. This class contains $3(\ell^2-1)$ isogenies.

			\item[$(\Nd)$] all other cases, where $G_\va=\{\pm I_2\} $  and $|O_\va|=36$. This class contains $\ell^3- 3\ell^2- 3\ell+ 5$ isogenies if $\ell\equiv 1\bmod{3}$ and $\ell^3- 3\ell^2- 3\ell+ 1$ isogenies if $\ell\equiv 2\bmod{3}$.
		\end{enumerate}	
	\end{prop}

	\subsection{Loops at $[E_{0}\times E_{0}]$}	
	\begin{thm}\label{thmE17281}
		Suppose $p>3\ell$. 
		
		$(1)$ If  $\ell \equiv 1 \bmod 3$, then  the set of loops of $E_{0} \times E_{0}$ is the union of Classes $\Da, \Db$ and $\Na$.
		
		$(2)$  If $\ell \equiv 2 \bmod 3$, then  the set of loops of $E_{0} \times E_{0}$ is Class $\Na$.					
		
		$(3)$ If $\ell=3$, then $E_{0} \times E_{0}$ has $1$ loop.		
	\end{thm}

	\subsection{Neighbors of $[E_{0}\times E_{0}] $} 
	
	If $\ell\equiv 1\bmod{3}$,	we denote the loops  $\lambda\times \lambda$ and $\bar{\lambda}\times \bar{\lambda}$ on $E_{0}\times E_{0} $ by   $[\lambda]$ and $[\bar{\lambda}]$ respectively. Hence
	\[ \ker[\lambda]=L_1\times L_1,\quad \ker[\bar{\lambda}]=L_2\times L_2. \]
	
	\begin{lemma}\label{goodiso1} Suppose $\ell\equiv 1\bmod{3}$ and $p>3\ell$. 
		For a loop $\alpha\neq [\lambda]$, $[\bar{\lambda}]$ of  $E_{0}\times E_{0}$ of degree $\ell^2$, let
		\[ L_{i}^\al=(L_i\times L_i)\cap \ker\al\ \ (i=1, 2). \]  
		Then $\alpha\mapsto L_i^{\alpha}$ gives a one-to-one correspondence 
		\[ \{\text{degree $\ell^2$ loops}\ \neq[\lambda]\ \text{or}\ [\bar{\lambda}]\}\leftrightarrow \{\text{$1$-dimensional subspaces of } L_i\times L_i\}.\]
		Consequently, $\ker\al=L_1^\al\times L_2^\al$, and if $L_1^\al=\{(P, kP): P\in L_1\}$ for some $k\neq 0$, then $L_2^\al=\{(-kQ, Q): Q\in L_2\}$.   
	\end{lemma}
	
	\begin{prop}\label{llc1}
		Suppose $p>3\ell^2$. Every loop $\va$ of $E_{0}\times E_{0}$ of degree $\ell^4$ corresponds to a matrix of the following form:
		\begin{enumerate}[I]
			\item $\ell I_2$. In this case $\ker\va\cong (\Z/\ell\Z)^4$.  
			\item $\begin{pmatrix}
				a & b\frac{1+ i}{2} \\
				\bar{b}\frac{{-1+ i}}{2}  & \bar{a} 
			\end{pmatrix}$, where $a, b\in \Z[\frac{1+ i}{2}]$, $a\bar{a}+ b\bar{b}= \ell^2$ and $ab\neq 0$.   In this case $\ker\va\cong (\Z/\ell^2\Z)^2$.
			\item $\lambda \begin{pmatrix}
				a & b\frac{1+ i}{2}\\
				\bar{b}\frac{-1+ i}{2} & \bar{a}
			\end{pmatrix}$ or $ \bar{\lambda} \begin{pmatrix}
				a & b\frac{1+ i}{2}\\
				\bar{b}\frac{-1+ i}{2} & \bar{a}
			\end{pmatrix}$, where $a, b\in \Z[\frac{1+ i}{2}]$ and $a\bar{a}+ b\bar{b}= \ell$.  In this case $\ker\va\cong (\Z/\ell\Z)^2\times \Z/\ell^2\Z$ which occurs only if $\ell\equiv 1\bmod{3}$.
		\end{enumerate}
		
		As a consequence, every loop of degree $\ell^4$ can be factorized as the product of two loops of degree $\ell^2$. Moreover, loops in Case $(2)$ is uniquely factorized as the composition of two edges of degree $\ell^2$. 
	\end{prop}
	
	\begin{lemma}\label{llc21}
		Suppose  $p> 3\ell^2$ and $\ell \equiv 1 \pmod 3$.
		
		If  $\al$ and $\beta$ are two edges from $E_{0}\times E_{0}$ to a vertex $V\neq E_{0}\times E_{0}$ such that  the loop $\va=\hat{\al}\circ \beta$ has a factor of $[\lambda]$ or $[\bar{\lambda}]$. Then
		\begin{enumerate}[1]
			\item $\al$ and $\beta$ are not in the same $G$-orbit.
			\item If $\va=\tau\circ [\lambda]$, then $\hat{\va}=\hat{\beta}\circ \al= \hat{\tau}\circ [\bar{\lambda}]$.
			\item If $\va=\tau\circ  [\lambda]$, then
			\begin{equation} \label{eq:kerg111} \ker(\beta)\cap (L_1\times L_1)= L_1^\tau,\quad  \ker(\al)\cap (L_2\times L_2)= L_2^{\hat{\tau}}. \end{equation}
			Hence $\tau$, $\al$ (resp. $\tau$, $\beta$) and  $\va$ are uniquely determined by $\beta$ (resp. $\al$).	
		\end{enumerate}
		
		On the other hand, if an edge $\beta: E_{0}\times E_{0}\rightarrow V$ of degree $\ell^2$ satisfies
		\begin{equation} \label{eq:kerg11} \dim \ker(\beta)\cap (L_i\times L_i)=1, \end{equation}
		for $i=1$ (resp. $i=2$), then there exists $\al:  E_{0}\times E_{0}\rightarrow V$ such that $\va=\hat{\al}\circ \beta=\tau\circ [\lambda]$  (resp. $\va=\hat{\al}\circ \beta=\tau\circ [\bar{\lambda}]$)  for some loop $\tau$.
	\end{lemma}

	\begin{lemma} \label{lemma:kerg1} Suppose $\ell\equiv 1\bmod{3}$ and $p>3\ell^2$. 
		\begin{enumerate}[1]
			\item The diagonal non-loop isogenies satisfying \eqref{eq:kerg11} are exactly those in Class $\Dc$, whose number is $4(\ell-1)$.
			\item The quadruples $\{a,b,c,d\}$ associated to  non-diagonal non-loop isogenies  satisfying \eqref{eq:kerg11} for $i=1$ belong to Class $\Nc$ and $\Nd$. 
			The number of isogenies satisfying \eqref{eq:kerg11}  for $i=1,2$ is $2(\ell-1)^2$, of which $12(\ell-1)$ are in Class $\Nc$ and $2(\ell-1)(\ell-7)$  in Class $\Nd$.
		\end{enumerate}
	\end{lemma}

	\begin{thm} \label{thm:trans1} Suppose $p>3\ell^2$.  Let $V$ be a vertex adjacent to $E_{0}\times E_{0}$.
		\begin{enumerate}[1]
			\item If there exists one isogeny $\beta: E_{0}\times E_{0}\rightarrow V$ satisfying \eqref{eq:kerg11}, then all isogenies from $E_{0}\times E_{0}$ to $V$ satisfy \eqref{eq:kerg11} and they form  two $G$-orbits.  
			\item For all other cases, $(\ell,\ell)$-isogenies from  $E_{0}\times E_{0}$ to $V$ form a $G$-orbit. 
		\end{enumerate}
	\end{thm}
	
	\begin{thm}\label{main1}
		Suppose $p>3\ell^2$. Consider the neighbor of $[E_0\times E_0]$.
		
		$(1)$ If  $\ell \equiv 1 \bmod 3$, the neighbor is given by the following   table:
		
		\begin{center}
			\begin{tabular}{|c|c|c||c|c|c|}
				\hline
				\#{Vertices} & Multi-Edges & Edge Type	& \#{Vertices} & Multi-Edges & Edge Type	  \\
				\hline
				$\frac{\ell-1}{3}$ & 12 &$\Dc$	 & $\frac{(\ell-1)(\ell-3)}{6}$ & 18 & $\Nc$-1\\
				\hline
				$\frac{\ell-1}{3}$ & 9 &$\Dd$	  & $\frac{\ell-1}{3}$ & 36 & $\Nc$-2 \\ 
				\hline
				$\frac{(\ell-1)(\ell-4)}{18}$ & 18 &$\De$	 & $\frac{(\ell-1)(\ell^2-4\ell+9)}{36}$ & 36 & $\Nd$-1 \\
				\hline
				$\frac{\ell-1}{6}$ & 6 & $\Nb$	 & $\frac{(\ell-1)(\ell-7)}{36}$ & 72  & $\Nd$-2  \\
				\hline
			\end{tabular}
		\end{center}

		$(2)$ If $\ell \equiv 2 \bmod 3$, the neighbor is given by the following   table:
		
		\begin{center}
			\begin{tabular}{|c|c|c||c|c|c|}
				\hline
				\#{Vertices} & Multi-Edges &  Edge Type & \#{Vertices} & Multi-Edges & Edge Type	\\
				\hline
				$\frac{\ell+1}{3}$ & 9 &$\Dd$	 & $\frac{\ell+1}{6}$ & 6 & $\Nb$\\
				\hline
				$\frac{\ell^2- \ell- 2}{18}$ & 18 & $\De$	 &  $\frac{\ell^2-1}{6}$  &  18 & $\Nc$ \\
				\hline
				&  &   &  $\frac{\ell^3- 3\ell^2- 3\ell+ 1}{36}$  & 36 & $\Nd$ \\
				\hline
			\end{tabular}
		\end{center}

		$(3)$ if $\ell=2$, then there is one vertex adjacent
		to $[E_{0}\times E_{0}]$ with diagonal kernel, and each connecting $[E_{0}\times E_{0}]$ with $9$ edges. There is one vertex adjacent
		to $[E_{0}\times E_{0}]$ with nondiagonal kernel, and each connecting $[E_{0}\times E_{0}]$ with $3$ edges, other three isogenies are loops.
		
		$(4)$ if $\ell=3$, then there are  two vertices adjacent
		to $[E_{0}\times E_{0}]$ with diagonal kernel, and connecting $[E_{0}\times E_{0}]$ with $6$ and $9$ edges. There are  three vertices adjacent
		to $[E_{0}\times E_{0}]$ with nondiagonal kernel, and each connecting $[E_{0}\times E_{0}]$ with $8$ edges, the last isogeny is  a loop.	
	\end{thm}
	
	\section{Loops and neighbors of $[E\times E']$ for $j(E)\in \F_p\backslash \{0, 1728\}$}
	In this section, we assume  $E$ has $j$-invariant in $\F_p\backslash \{0, 1728\}$ and $j(E')\notin\{0,1728,j(E)\}$. Let $\pi$ be the Frobenius map of $E$. We know that $\End(E)$ has the form (see for example \cite{LOX201}): 
	
	\begin{enumerate} 
		\item  If $\frac{1+\pi}{2} \notin {\End}(E)$, then there exists a prime $q$ satisfying $(\frac{p}{q})=-1$ and  $q \equiv 3 \pmod 8$, such that
		\[{\End}(E)=\gO(q):=\Z+\Z \frac{1+i}{2}+\Z \frac{j-k}{2}+\Z \frac{ri-k}{q}, \]
		where   $r^2 \equiv -p \pmod q$,  $i^2=-q$,  $j^2=-p$,  $ij=-ji=k$.
		The reduced norm of $\gO(q)$ is given by
		\begin{equation} \begin{split}
				&\Nrd\left(x+y\frac{1+ i}{2}+z \frac{j- k}{2}+ w\frac{ri- k}{q}\right)\\ &= \left(x+\frac{y}{2}\right)^2+ q\left(\frac{y}{2}+\frac{rw}{q}\right)^2+ \frac{pz^2}{4}+ pq\left(\frac{z}{2}+ \frac{w}{q}\right)^2. 	
			\end{split}
		\end{equation} 
		
		\item If $\frac{1+\pi}{2} \in {\End}(E)$, then there exists a prime $q$ satisfying $(\frac{p}{q})=-1$ and  $q \equiv 3 \pmod 8$, such that	 
		\[ {\End}(E)=\gO'(q):=\Z+\Z \frac{1+j}{2}+\Z i+\Z \frac{r'i-k}{2q}, \]
		where  ${r'}^2 \equiv -p \pmod {4q}$,  $i^2=-q$,  $j^2=-p$,  $ij=-ji=k$. The reduced norm on $\gO(q)$ is given by
		\begin{equation} \begin{split}
				&\Nrd\left(x+yi+z \frac{1+ j}{2}+ w\frac{r'i- k}{2q}\right)\\ &= \left(x+\frac{z}{2}\right)^2+ q\left(y+\frac{r'w}{2q}\right)^2+ \frac{p(qz^2+ w^2)}{4q}.	
			\end{split}
		\end{equation} 
	\end{enumerate}

	\subsection{Kernels of $(\ell,\ell)$-isogenies from  $E\times E'$}
	
	Note that $E, E'$ are not isomorphic, we have
	\[ G= \Aut(E\times E')=\{1, [-1]\times 1, 1\times [-1], [-1]\times [-1]\}\cong \Z/2\Z\times \Z/2\Z. \]
	Consequently, we have
	\begin{prop}\label{pee}  Let $\va$ be  an $(\ell, \ell)$ isogeny  from $E\times E'$. 
		\begin{enumerate}[1]
			\item If $\ker\va$ is diagonal (type $(\DD)$), then $G_\va= G$ and $O_\va=\{\va\}$. 
			
			\item If $\ker\va$ is nondiagonal (type $(\NNN)$), then $G_\va= \{1, [-1] \times 1\}$ and $|O_\va|= 2$. 	 		
		\end{enumerate}
	\end{prop}	
	
	\subsection{Loops at $[E\times E']$}

	\begin{thm}\label{thmee} 
		For $d\in\Z_+$, let $\Iso_d(E,E'):=\{\sigma: E\rightarrow E'\mid \deg(\sigma)=d\}$. 
		Suppose either $\End(E)= \gO(q)$ and $p> q\ell^2> 4\ell^4$ or  $\End(E)= \gO'(q)$ and $p> 4q\ell^2> 4\ell^4$.

		\begin{enumerate}[1]
			\item  If there exists $d$ such that $\ell-d=\square>0$ (where $\square$ denotes a square of an integer) and $\Iso_d(E,E')\neq \emptyset$, then there are exactly two loops of $E\times E'$, whose kernels are nondiagonal.  
			
			\item If there is an isogeny from $E$ to $E'$ of degree $\ell$, then there is only one loop of $E\times E'$, whose kernel is diagonal.  
			
			\item  If  $\Iso_d(E,E')=\emptyset$ for all $d$ such that $\ell-d=\square$, then there is no loop of $E\times E'$.
		\end{enumerate} 
	\end{thm}
	
	\begin{proof} As seen in \cite[\S 2]{MM}, a loop of $E\times E'$ has the form: 
		\[\begin{pmatrix}
			\va_1 & -\va_2 \\
			\va_3 & \va_4
		\end{pmatrix}\] 
		where $\va_1:E\rw E$, $\va_2: E' \rw E$, $\va_3: E \rw E'$, $\va_4: E' \rw E'$, and \[\begin{pmatrix}
			\widehat{\va_1} & \widehat{\va_3}  \\
			-\widehat{\va_2}& \widehat{\va_4}
		\end{pmatrix}\begin{pmatrix}
			\va_1 & -\va_2 \\
			\va_3 & \va_4
		\end{pmatrix} = \begin{pmatrix}
			\ell & 0\\
			0 & \ell
		\end{pmatrix}.\] 
		By computation, we have
		\begin{equation} \label{eq:va1234} \begin{split}
				&\deg(\va_1)= \deg(\va_4),\quad \deg(\va_2)= \deg(\va_3),\\ &\deg(\va_1)+ \deg(\va_2)= \ell,\quad \widehat{\va_2}\circ \va_1 = \widehat{\va_4}\circ \va_3.
		\end{split} \end{equation} 
		
		We now show the case $\End(E)=\gO(q)$, the case  $\End(E)=\gO'(q)$ follows by the same argument. 
		
		Under the assumption $p> q\ell^2> 4\ell^4$, loops of $E$ of degree  $\leq \ell$ are all inside $\Z$. This means $\va_1= [a]$ for some $a\in \Z$. By \eqref{eq:va1234}, we have $\gcd(\deg(\va_1), \deg(\va_2))=1$, and  $E[a]\su \ker(\widehat{\va_2}\circ\va_1)$. Thus $E[a]\su \ker(\widehat{\va_4}\circ\va_3)$, which implies that $\ker(\widehat{\va_4})\cong \left(\Z/a\Z \right)^2$. Hence $\va_4=[a]$, $\va_2= \widehat{\va_3}$ and $a^2+ \deg(\va_2)= \ell$. Thus if $\begin{pmatrix}
			\va_1 & -\va_2 \\
			\va_3 & \va_4
		\end{pmatrix}$ is a loop of $E\times E'$, $\va_2: E \rw E'$ must be an isogeny such that $\ell-\deg(\va_2)=\square$. This gives the proof of (3).
		
		(1) If there is an isogeny $\va$ from $E$ to $E'$ of degree $d<\ell$ such that  $\ell-d=a^2$, then there are two loops of $E\times E'$: 	 
		\[\begin{pmatrix}
			a & -\widehat{\va} \\
			\va & a
		\end{pmatrix}, \quad \begin{pmatrix}
			a & \widehat{\va} \\
			-\va & a
		\end{pmatrix}. \]
		The kernels of these two loops are: \[ \la ([a]P, -\va(P)), ([a]Q, -\va(Q)) \ra; \quad \la ([a]P, \va(P)), ([a]Q, \va(Q)) \ra \]
		where $E[\ell]= \la P, Q\ra$. It is easy to see these two kernels are nondiagonal and in the same orbit of $G= \Aut(E \times E')$. 
		
		We show that $\va$ is  the unique  isogeny from $E$ to $E'$ such that $\deg(\va)\leq \ell$ and $\ell-\deg(\va)=\square$. If there is another isogeny $\va'$ from $E$ to $E'$ of degree $d'\le \ell$ such that $\ell-d'=b^2$, then $\widehat{\va'}\circ \va$ is a loop of $E$ of degree $dd'$. By the assumption $p> q\ell^2>4\ell^4$,  we have $\widehat{\va'} \circ \va = [x]$ with $x \in \Z$. Hence $[\deg(\va')]\varphi=[x]\va'$. Comparing the degrees on both side, we find that there exist coprime integers $t_1$ and $t_2$ such that $t_1^2 d'=t_2^2 d$, which means $\delta=d/t_1^2=d'/t_2^2\in \Z_+$. By the relation $\ell= d+ a^2= d'+b^2$, we have
		\[ \ell= t_1^2 \delta+a^2= t_2^2\delta+b^2. \]   
		Thus $\ell$ splits into principal prime ideals in the imaginary quadratic field $\Q(\sqrt{-\delta})$. By unique factorization,  we must have $d= d'$ and $\va= \va'$. 
		
		(2) If there is an isogeny $\va$ from $E$ to $E'$ of degree $\ell$, then there is one loop 
		$\begin{pmatrix}
			0 & -\widehat{\va} \\
			\va & 0
		\end{pmatrix}$ 
		of $E\times E'$, whose kernel 	  is diagonal. 
		
		We show that $\va$ is the unique  isogeny from $E$ to $E'$ such that $\deg(\va)\leq \ell$ and $\ell-\deg(\va)=\square$. 	Indeed, suppose $\va'$ is another isogeny satisfying this property,  then $\widehat{\va'}\circ \va=[x]$ for some $x\in \Z$. If $\deg(\va')=\ell$, then $\va=\va'$. The argument in (i) means the case $\deg(\va')<\ell$ is impossible. 
	\end{proof} 
	
	\subsection{ Neighbors of $[E\times E'] $ }
	
	\begin{thm}
		Suppose either $\End(E)= \gO(q)$ and $p> q\ell^2> 4\ell^4$ or  $\End(E)= \gO'(q)$ and $p> 4q\ell^2> 4\ell^4$. 
		\begin{enumerate}[1]
			\item  If there is an isogeny from $E$ to $E'$ of degree $d$ such that $\ell- d= \square>0$, then the neighbor   of $[E\times E'] $ is given by the following  table:
			
			\begin{center}
				\begin{tabular}{|c|c|c||c|c|c|}
					\hline
					\#{Vertices} & Multi-Edges & 	Edge Type	& \#{Vertices} & Multi-Edges & Edge Type	\\
					\hline
					$(\ell+ 1)^2$ & 1 & $\DD$		& $\frac{\ell^3- \ell- 2}{2}$ & 2 & $\NNN$  \\
					\hline
				\end{tabular}
			\end{center}
			
			\item If there is an isogeny from $E$ to $E'$ of degree $\ell$, then the neighbor  is given by the following  table:
			
			\begin{center}
				\begin{tabular}{|c|c|c||c|c|c|}
					\hline
					\#{Vertices} & Multi-Edges & 	Edge Type	& \#{Vertices} & Multi-Edges & Edge Type	\\
					\hline
					$\ell^2+ 2\ell$ & 1 & $\DD$		& $\frac{\ell^3- \ell}{2}$ & 2  & $\NNN$ \\
					\hline
				\end{tabular}
			\end{center}
			
			\item If there is no isogeny from $E$ to $E'$ of degree $d$ such that $\ell- d= \square$, then the neighbor  is given by the following  table:
			
			\begin{center}
				\begin{tabular}{|c|c|c||c|c|c|}
					\hline
					\#{Vertices} & Multi-Edges & 	Edge Type	& \#{Vertices} & Multi-Edges & Edge Type	\\
					\hline
					$(\ell+ 1)^2$ & 1 & $\DD$		& $\frac{\ell^3- \ell}{2}$ & 2 & $\NNN$ \\
					\hline
				\end{tabular}
			\end{center}
		\end{enumerate}
	\end{thm}
	
	\begin{proof}
		We claim that two $(\ell,\ell)$-isogenies $\al, \be$ from $[E \times E']$ to the same vertex are in the same $G$-orbit under the assumption of the theorem.
		
		Indeed, in this situation  $\widehat{\be}\circ \al$ is a loop of $[E\times E']$. Write \[\widehat{\be} \circ \al= \begin{pmatrix}
			\va_1 & -\va_2 \\
			\va_3 & \va_4
		\end{pmatrix}\] 
		where $\va_1:E\rw E$, $\va_2: E' \rw E$, $\va_3: E \rw E'$ and $\va_4: E' \rw E'$ are isogenies of elliptic curves.
		Under the assumption $p> q\ell^2>4\ell^4$ for $\End(E)=\gO(q)$ or  $p> 4q\ell^2>4\ell^4$  for $\End(E)=\gO'(q)$, loops of $E$ with degree $\le \ell^2$ are all inside $\Z$. Suppose $\va_1=[a]$ for some $a\in \Z$. Following the proof of Theorem~\ref{thmee} (1), we have
		\[ \va_1= \va_4=[a], \ \va_2= \widehat{\va_3} \text{ and } a^2+ \deg(\va_2)= \ell^2. \] 
		We now denote $\va_3$ by $\va$. Then 
		\[ \ker(\widehat{\be}\circ \al)=\la ([a]P, -\va(P)), ([a]Q, -\va(Q)) \ra\cong  \left(\Z/\ell^2\Z\right)^2\]
		where $E[\ell^2]= \la P, Q\ra$, thus $\be$ is a good extension of $\al$. 
		Hence, 
		\[\ker(\al)=\ker(\widehat{\be}\circ \al)  \cap (E \times E')[\ell]= \la ([a\ell]P, -\va([\ell]P)), ([a\ell]Q, -\va([\ell]Q)) \ra. \]
		Similarly
		\[ \ker(\be)=\ker(\widehat{\al}\circ\be)\cap (E \times E')[\ell]=\la ([a\ell]P, \va([\ell]P)), ([a\ell]Q, \va([\ell]Q)) \ra.  \]
		This implies that $\al, \be$ are in the same orbit. 
		
		By the above claim and Proposition~\ref{pee}, an adjacent vertex connects to $E\times E'$ by either  two edges whose kernels are nondiagonal or  only one edge whose kernel is diagonal. Now the theorem follows from Corollary~\ref{cor:isogenies} and Theorem~\ref{thmee}.	
	\end{proof}
	
	\section{Loops and neighbors of $[E\times E]$ for $j(E)\in \F_p\backslash \{0, 1728\}$}
	
	In this section, we assume $E$ is a  supersingular elliptic curve over $\F_p$ such that $j(E)\neq  0,\ 1728$. The maximal orders  $\gO(q)$ and $\gO'(q)$ are defined as in \S~5. Results in this section are parallel to results in \S\ 5, with almost identical proof which will be omitted.

	\subsection{Kernels of $(\ell,\ell)$-isogenies from  $E\times E$}
	
	Note that $E$ is a supersingular elliptic curve over $\F_{p^2}$, we have
	\[G:= \Aut(E\times E)=\{[\pm 1]\times  [\pm 1]\} \times \left\{1, \begin{pmatrix}
		0 & 1\\
		1 & 0
	\end{pmatrix}\right\}\cong \left(\Z/2\Z\right)^3. \]
	Consequently, we have
	\begin{prop}\label{peee}  
		The	$(\ell, \ell)$ isogenies from $E\times E$ can be divided into the following classes, which are unions of $G$-orbits: 
		\begin{enumerate}
			\item[$(\DD1)$]  $\ker(\va)= K_1\times K_2$, where  $K_1=K_2\su E[\ell]$ has order $\ell$.  In this case, $G_\va= G$ and $O_\va=\{\va\}$. This class contains $\ell+ 1$ isogenies.
			
			\item[$(\DD2)$]  $\ker(\va)= K_1\times K_2$, where  $K_1\ne K_2\su E[\ell]$ has order $\ell$.  In this case, $|G_\va|= 4$ and $|O_\va|= 2$. This class contains $\ell(\ell+1)$ isogenies.
			
			\item[$(\NNN1)$] if $\ell \equiv 1 \pmod 4$, and $\ell= a^2+ b^2$, $t= -\frac{a}{b}$ in $\Z/\ell\Z$, $\ker(\va)= \la (P, tP), (Q, tQ) \ra$ or $\la (P, -tP), (Q, -tQ) \ra$, $\la P, Q\ra= E[\ell]$. In this case, the isogeny $\va$ corresponds to one of the following two loops of $[E \times E]$: 
			\[\begin{pmatrix}
				a & b \\
				-b & a
			\end{pmatrix}, \quad \begin{pmatrix}
				a & b\\
				b & -a
			\end{pmatrix}.\]

			\item[$(\NNN2)$]  $\ker(\va)= \la (P, aP+ bQ), (Q, cP+ dQ) \ra$, where $ad- bc= -1$, $a+ d= 0$, $\la P, Q\ra= E[\ell]$.  In this case, $G_\va= \{1, [-1] \times [-1]\}\times \left\{ 1, \begin{pmatrix}
				0 & 1\\
				1 & 0\\
			\end{pmatrix} \right\}$ and $|O_\va|= 2$. This class contains $\ell(\ell+1)$ isogenies. 
			
			\item[$(\NNN3)$]  $\ker(\va)= \la (P, aP+ bQ), (Q, cP+ dQ) \ra$, where $ad- bc= -1$, $a+ d \ne 0$,  $\la P, Q\ra= E[\ell]$.  In this case, $G_\va= \{1, [-1] \times [-1]\}$ and $|O_\va|= 4$. 	This class contains $\ell^3- \ell^2- 2\ell- 2$ isogenies if $\ell \equiv 1 \pmod 4$, and $\ell^3- \ell^2- 2\ell$ isogenies if $\ell \equiv 3 \pmod 4$.	
		\end{enumerate}
	\end{prop}	
	\subsection{Local structure of $[E\times E]$}	
	
	\begin{thm}\label{thmeee} 
		Suppose either $\End(E)=\gO(q)$  and $p> q\ell> 4\ell^2$ or $\End(E)= \gO'(q)$ and $p> 4q\ell> 4\ell^2$. 
		
		\begin{enumerate}[1]
			\item  If $\ell\equiv 1 \pmod 4$, then there are exactly two loops of $E\times E$, whose kernels are in $(\NNN1)$.   
			
			\item  If $\ell \equiv 3 \pmod 4$, then there is no loop of $E\times E$.
		\end{enumerate} 
	\end{thm}

	\begin{thm} Suppose either $\End(E)=\gO(q)$  and $p> q\ell^2> 4\ell^4$ or $\End(E)= \gO'(q)$ and $p> 4q\ell^2> 4\ell^4$. 
		\begin{enumerate}[1]
			\item  If $\ell \equiv 1 \pmod 4$, the neighbor of $[E \times E]$ is given by the following  table:
			
			\begin{center}
				\begin{tabular}{|c|c|c||c|c|c|}
					\hline
					\#{Vertices} & Multi-Edges & Edge Type	& \#{Vertices} & Multi-Edges & Edge Type \\
					\hline
					$\ell+ 1$ & 1 & $\DD1$		& $\frac{\ell^2+ \ell}{2}$ & 2 & $\NNN2$ \\
					\hline
					$\frac{(\ell+ 1)\ell}{2}$ & 2 & $\DD2$	& $\frac{\ell^3- \ell^2- 2\ell- 2}{4}$ & 4  & $\NNN3$	\\
					\hline
				\end{tabular}
			\end{center}
			
			\item If $\ell \equiv 3 \pmod 4$, the neighbor of  is given by the following  table:
			
			\begin{center}
				\begin{tabular}{|c|c|c||c|c|c|}
					\hline
					\#{Vertices} & Multi-Edges & Edge Type	& \#{Vertices} & Multi-Edges & Edge Type	\\
					\hline
					$\ell+ 1$ & 1 & $\DD1$		& $\frac{\ell^2+ \ell}{2}$ & 2 & $\NNN2$\\
					\hline
					$\frac{(\ell+ 1)\ell}{2}$ & 2 & $\DD2$		& $\frac{\ell^3- \ell^2- 2\ell}{4}$ & 4  & $\NNN3$ \\
					\hline
				\end{tabular}
			\end{center}
		\end{enumerate}
	\end{thm}

	\section{A simple proof of Main Theorem in \cite{LOX20}}
	In this section, we give an alternative proof of the following theorem  in \cite{LOX20} without using Deuring's correspondence in \cite{d}. Our new proof is similar to the proof of Proposition ~\ref{llc}. 
	\begin{thm} Suppose $\ell>3$.  Consider the $\ell$-isogeny graph $\gG_{\ell}(\F_{p^2},-2p)$ of supersingular elliptic curves over $\F_{p^2}$ of trace $-2p$.
		
		$(1)$ If $p\equiv 3\bmod{4}$ and $p>4\ell^2$,  there are $\frac{1}{2}\bigl(\ell-(-1)^{\frac{\ell-1}{2}}\bigr)$ vertices adjacent to  $[E_{1728}]$ in the graph, each connecting $[E_{1728}]$ with $2$ edges.
		Moreover,  $1+(\frac{\ell}{p})$ of the vertices are of   $j$-invariants in $\F_p-\{1728\}$.
		
		$(2)$ If $p\equiv 2\bmod{3}$ and $p>3\ell^2$, there are $\frac{1}{3}(\ell-(\frac{\ell}{3}))$  vertices adjacent to $[E_{0}]$ in the graph, each connecting $[E_{0}]$ with $3$ edges. Moreover,  $1+(\frac{-p}{\ell})$ of the vertices are of   $j$-invariants in $\F_p^*$.
	\end{thm}
	\begin{proof} We will show the case $E_{1728}$, the case  $E_{0}$ follows by a parallel argument.
		
		Suppose  $p \equiv 3 \pmod 4$. Then $\Aut(E_{1728})=\{\pm 1, \pm i\}$ where $i^2=-1$. Suppose $\ell$ is another prime. If  $p>4\ell^2$, then elements in $\gO$ of reduced norm $\ell^2$ are actually inside $\Z[i]$. Thus loops from $E_{1728}$ of degree $\ell^2$ lies in $\Z[i]$. 
		
		If $\ell\equiv 3\pmod 4$, the only elements in $\Z[i]$ of reduced norm $\ell^2$ are $\sigma \ell$  where $\sigma\in \Aut(E_{1728})$, so $\ell$ is only one loop of degree $\ell^2$. If there are two different isogenies $\va, \psi$ from $E_{1728}$ of degree $\ell$ to the same adjacent vertex, then $\widehat{\psi}\circ \va$ is a loop from $E_{1728}$, which is  $\ell$. Hence $\va= \psi\circ\sigma$ with $\sigma\in \Aut(E_{1728})$. Since $\va, \psi$ are not loops, we have $\ker(\psi)= i(\ker(\va)) \ne \ker(\va)$. It means an adjacent vertex connects $E_{1728}$ with two edges. 
		
		If $\ell\equiv 1\pmod 4$, let $x,y\in \Z$ such that $x^2+y^2=\ell$. Then elements  of reduced norm  $\ell$ in $\Z[i]$ are $\sigma(x\pm  yi)$  where $\sigma\in \Aut(E_{1728})$, so  the  loops with degree $\ell$ from $E_{1728}$ are $x\pm yi$. Elements of reduced norm $\ell^2$ are $\sigma (x\pm  yi)^2$ and $\sigma \ell$, corresponding to the three loops $(x\pm yi)^2$ and $\ell$ of degree $\ell^2$ on $E_{1728}$. Furthermore, the first two loops are compositions of a loop of degree $\ell$ with itself. Hence if there are two different isogenies $\va, \psi$ from $E_{1728}$ with degree $\ell$ to the same adjacent vertex, then $\widehat{\psi}\circ \va$ is a loop of $E_{1728}$. Since $(x\pm yi)^2 \ne \ell$,  we have $\ker((x\pm yi)^2)$ is cyclic (i.e $\ker((x \pm yi)^2)\cong \Z/\ell^2\Z$). If $\widehat{\psi} \circ \va= (x+ yi)$ or $(x- yi)^2$, then $\ker(\va)= \ker(x+ yi)$ or $\ker(x-yi)$ which means $\va$ is a loop.  Hence we have $\widehat{\psi} \circ \va$ is not equal to $(x+ yi)^2$ or $(x- yi)^2$. This implies the loop is $\ell$, and $\va= \psi\circ\sigma$ where $\sigma\in \Aut(E_{1728})$. Since $\va, \psi$ are not loops, we have $\ker(\psi)= i(\ker(\va)) \ne \ker(\va)$. It means an adjacent vertex connects $E_{1728}$ with two edges. 
	\end{proof}
	
	\section*{Acknowledgment}
	
	Zheng Xu and Yi Ouyang was supported by Innovation Program for Quantum Science and Technology (Grant No. 2021ZD0302902), NSFC(Grant No. 12371013), Anhui Initiative in Quantum Information Technologies (Grant No. AHY150200). Zijian Zhou was supported by NSFC(Grant No.62202475).

\end{document}